\RequirePackage[reqno]{amsmath}
\documentclass[a4paper]{amsart}

\usepackage[english]{babel}

\usepackage{amsfonts}
\usepackage{amssymb}
\usepackage{amsthm}
\usepackage{amsmath}
\usepackage{upgreek}
\usepackage{indentfirst}
\usepackage{empheq}

\usepackage[colorlinks=false,backref=false,pagebackref=false,pdftex,pdfauthor={Vanderley, Filippo, Ederson},pdftitle={To decide}]{hyperref}

\usepackage[utf8]{inputenc}

\usepackage{setspace}
\usepackage[color=green]{todonotes}
\usepackage{graphicx}

\usepackage{color}
\usepackage{enumerate}

\newtheorem{theorem}{Theorem}
\newtheorem{lemma}[theorem]{Lemma}
\newtheorem{corollary}[theorem]{Corollary}
\newtheorem{proposition}[theorem]{Proposition}

\newtheorem{definition}[theorem]{Definition}
\newtheorem{remark}[theorem]{\bf Remark}

\usepackage{mathtools}
\mathtoolsset{showonlyrefs}
\usepackage{xcolor}

\begin{document}

\title{Instability of modes in a partially hinged rectangular plate}

\author[Ferreira Jr]{Vanderley Ferreira Jr}
\address{Vanderley Ferreira Jr and Ederson Moreira dos Santos\newline \indent Instituto de Ci{\^e}ncias Matem{\'a}ticas e de Computa\c{c}{\~a}o --- Universidade de S{\~a}o Paulo \newline \indent
Caixa Postal 668, CEP 13560-970 - S\~ao Carlos - SP - Brazil}
\email{vanderley.cn@gmail.com, \ \ ederson@icmc.usp.br}

\author[Gazzola]{Filippo Gazzola}
\address{Filippo Gazzola \newline \indent Dipartimento di Matematica --- Politecnico di Milano \newline
\indent Piazza Leonardo da Vinci 32, 20133 Milano, Italy}
\email{filippo.gazzola@polimi.it}

\author[Moreira dos Santos]{Ederson Moreira dos Santos}

\date{\today}


\begin{abstract}
We consider a thin and narrow rectangular plate where the two short edges are hinged whereas the two long edges are free. This plate aims to represent the deck of a bridge, either a footbridge or a suspension bridge. We study a nonlocal evolution equation modeling the deformation of the plate and we prove existence, uniqueness and asymptotic behavior for the solutions for all initial data in suitable functional spaces. Then we prove results on the stability/instability of {\em simple modes} motivated by a phenomenon which is visible in actual bridges and we complement these theorems with some numerical experiments.
\end{abstract}

\maketitle

\tableofcontents

\section{Introduction}

We consider a thin and narrow rectangular plate where the two short edges are hinged whereas the two long edges are free.
This plate aims to represent the deck of a bridge, either a footbridge or a suspension bridge. In absence of forces, the plate lies flat
horizontally and is represented by the planar domain $\Omega=(0,\pi)\times(-l,l)$ with $0<l\ll\pi$.
The plate is subject to dead and live loads acting orthogonally on $\Omega$: these loads can be either pedestrians, vehicles, or the
vortex shedding due to the wind. The plate is also subject to edge loads, also called buckling loads, that are compressive forces along the edges:
this means that the plate is subject to prestressing.\par
We follow the plate model suggested by Berger \cite{berger}; see also the previous beam model suggested by Woinowsky-Krieger \cite{woinowsky} and,
independently, by Burgreen \cite{burg}. Then, the nonlocal evolution equation modeling the deformation of the plate reads
\begin{empheq}{align}
\label{enlm}
\left\{
\begin{array}{rl}
U_{tt}+\delta U_t + \Delta^2 U - \phi(U) U_{xx}=F  &\textrm{in }\Omega\times(0,T)\\
U = U_{xx}= 0 &\textrm{on }\{0,\pi\}\times[-l,l]\\
U_{yy}+\sigma U_{xx} = U_{yyy}+(2-\sigma)U_{xxy}= 0 &\textrm{on }[0,\pi]\times\{-l,l\}\\
U(x,y, 0) = U_0(x,y), \quad \quad U_t(x,y, 0) = V_0(x,y) &\textrm{in }\Omega
\end{array}
\right.
\end{empheq}
where the nonlinear term $\phi$ is defined by
\begin{empheq}{align}\label{nonlocalterm}
\phi (U) = -P + S \int_\Omega U_x^2\,,
\end{empheq}
and carries a nonlocal effect into the model. Here $S>0$ depends on the elasticity of the material composing the deck, $S\int_\Omega U_x^2$
measures the geometric nonlinearity of the plate due to its stretching, and $P>0$ is the prestressing constant: one has $P>0$ if the plate is
compressed and $P<0$ if the plate is stretched. The constant $\sigma$ is the Poisson ratio: for metals its value lies around $0.3$ while for
concrete it is between $0.1$ and $0.2$. We assume throughout this paper that
\begin{equation}\label{conditionsigma}
0<\sigma<\frac{1}{2}.
\end{equation}
The function $F: \Omega \times [0,T]\rightarrow \mathbb{R}$ represents the vertical load over the deck and may depend on time while $\delta$
is a damping parameter. Finally $U_0$ and $V_0$ are, respectively, the initial position and velocity of the deck.
The boundary conditions on the short edges are named after Navier \cite{navier} and model the fact that the plate is hinged in connection with
the ground; note that $U_{xx}=\Delta U$ on $\{0,\pi\}\times(-l,l)$. The boundary conditions on the long edges model the fact that the plate is free;
they may be derived with an integration by parts as in \cite{mansfield,ventsel}. For a partially hinged plate such as $\Omega$, the buckling load only
acts in the $x$-direction and therefore one obtains the term $\int_\Omega U_x^2$; see \cite{knightly}. We refer to
\cite{2014al-gwaizNATMA,2015ferreroDCDSA,2015gazzolamodeling} for the derivation of \eqref{enlm}, to the recent monograph \cite{2015gazzola}
for the complete updated story, and to \cite{villaggio} for a classical reference on models from elasticity. The behavior of rectangular plates
subject to a variety of boundary conditions is studied in \cite{braess,grunau,gruswe,gruswe2,sweers}.\par
The first purpose of the present paper is to prove existence, uniqueness and asymptotic behavior for the solutions of \eqref{enlm} for all
initial data in suitable functional spaces. We state and discuss these results in Section \ref{section:weel-posedeness} and their proofs are presented in Sections \ref{section:proofEU} and \ref{section:asymptotic}. We will mainly deal with weak solutions, although with little effort one could extend the results
to more regular solutions (including classical solutions) by arguing as in the seminal paper by Ball \cite{ball} for the beam equation.
By separating variables, we show that \eqref{enlm} admits solutions with a finite number of nontrivial Fourier components. This enables us to
define the (nonlinear) {\em simple modes} of \eqref{enlm}: we point out that, contrary to linear equations, the period of a nonlinear mode depends on the amplitude
of oscillation. The simple modes are found by solving a suitable eigenvalue problem for the stationary equation, which is the subject treated in Section \ref{section:eigenvalue}. In this respect, we take
advantage of previous results in \cite{bfg1,2015ferreroDCDSA,2015gazzolamodeling} where the main properties of the eigenfunctions were studied. In particular,
it was shown that the eigenfunctions may be classified in two distinct families: one family contains the so-called longitudinal eigenfunctions which,
approximately, have the shape of $c_m\sin(mx)$, the other family contains the so-called torsional eigenfunctions which, approximately, have the shape
of $c_my\sin(mx)$.\par
In \cite{bfg1}, an attempt to study the stability of the (nonlinear) simple modes for a local equation similar to \eqref{enlm} is performed.
It turns out that local problems do not allow separation of variables and the precise characterization of the simple modes. The results in
\cite{bfg1} show that there is very strong interaction between these modes and that, probably, the local version of the equation \eqref{enlm}
needs to be further investigated. A similar difficulty appears for the nonlinear string equation: for this reason, Cazenave-Weissler
\cite{cazw,1996cazenaveQAM} suggest to deal first in full detail with the stability of the modes in the nonlocal version.
The second purpose of the present paper is precisely to study the stability of the simple modes of \eqref{enlm}. We collect our results on this subject in Section \ref{section:theorems-stability}, we prove them in Section \ref{section:stability} and they are complemented with some numerical experiments in Section \ref{section:numerics}. This study is motivated by
a phenomenon which is visible in actual bridges and we mention that, according to the Federal Report \cite{ammann} (see also \cite{scott}), the main reason for the
Tacoma Narrows Bridge collapse was the sudden transition from longitudinal to torsional oscillations. Several other bridges collapsed for the same reason,
see e.g.\ \cite[Chapter 1]{2015gazzolamodeling} or the introduction in \cite{bfg1}.\par
In his celebrated monograph, Irvine \cite[p.176]{irvine} suggests to
initially ignore damping of both structural and aerodynamic origin in any model for bridges. His purpose is to simplify as much as possible the model
by maintaining its essence, that is, the conceptual design of bridges. Since the origin of the instability is of structural nature (see \cite{2015gazzola}
for a survey of modeling and results), in this paper we follow this suggestion: we analyze in detail how a solution of \eqref{enlm} initially oscillating
in an almost purely longitudinal fashion can suddenly start oscillating in a torsional fashion, even without the interaction of external forces, that is,
when $F\equiv0$. Overall, our results fit qualitatively with the description of the instability appeared during the Tacoma collapse. The interactions of
the deck (plate) with other structural components (cables, hangers, towers), as well as the aerodynamic and damping effects, are fairly important in
actual bridges: we left all of them out of our model. But we expect that if some phenomena arise in our simple plate model, then they should also be
visible in more complex structures and sophisticated models.

\section{Longitudinal and torsional eigenfunctions} \label{section:eigenvalue}

Throughout this paper we deal with the functional spaces
\begin{empheq}{align}
H^{2}_*(\Omega)=\{U\in H^2(\Omega); \ U=0 \ \textrm{on }\{0,\pi\}\times[-l,l]\}\,,
\end{empheq}
\begin{empheq}{align}
H^{1}_*(\Omega) = \{U\in H^1(\Omega); \ U=0\ \textrm{on }\{0,\pi\}\times[-l,l]\}\,,
\end{empheq}
as well as with $\mathcal{H}$, the dual space of $H^{2}_*(\Omega)$.
We use the angle brackets $\langle{\cdot, \cdot}\rangle$ to denote the duality of $\mathcal{H}\times H^{2}_*(\Omega)$, $({\cdot, \cdot})_0$ for the inner product in $L^2(\Omega)$
with $\|{ \cdot }\|_0$ the standard norm in $L^2(\Omega)$, $({U,V})_1= \int_{\Omega} \nabla U \nabla V $ for the inner product in $H^{1}_*(\Omega)$,
$({\cdot \, , \cdot})_2$ for the inner product in $H^{2}_*(\Omega)$ defined by
\begin{empheq}{align}({U,V})_2\!=\!\int_{\Omega}\left( \Delta U \Delta V\!-\!(1\!-\!\sigma)\big(U_{xx}V_{yy}\!+\!U_{yy}V_{xx}\!-\!2U_{xy}V_{xy}
\big) \right), \quad U,V\in H^{2}_*(\Omega)\,.
\end{empheq}
Thanks to assumption \eqref{conditionsigma}, this inner product defines a norm in $H^{2}_*(\Omega)$; see \cite[Lemma 4.1]{2015ferreroDCDSA}.\par
Our first purpose is to define a suitable basis of $H^{2}_*(\Omega)$ and to define what we mean by simple modes of \eqref{enlm}. To this end, we consider
the eigenvalue problem
\begin{equation}\label{eq:eigenvalueH2L2}
\left\{
\begin{array}{rl}
\Delta^2 w = \lambda w& \textrm{in }\Omega \\
w = w_{xx} = 0& \textrm{on }\{0,\pi\}\times[-l,l]\\
w_{yy}+\sigma w_{xx} = 0& \textrm{on }[0,\pi]\times\{-l,l\}\\
 w_{yyy}+(2-\sigma)w_{xxy} = 0& \textrm{on }[0,\pi]\times\{-l,l\}
\end{array}
\right.
\end{equation}
which can be rewritten as $(w,z)_2=\lambda (w,z)_0$ for all $z\in H^{2}_*(\Omega)$. From \cite{bfg1,2015ferreroDCDSA} we know that
the set of eigenvalues of \eqref{eq:eigenvalueH2L2} may be ordered in an increasing sequence of strictly positive numbers diverging
to $+\infty$, any eigenfunction belongs to $C^\infty(\overline\Omega)$, and the set of eigenfunctions of \eqref{eq:eigenvalueH2L2} is a complete
system in $H^2_*(\Omega)$. In fact, for any $m\ge1$ there exists a divergent sequence of eigenvalues $(\lambda_{m,i})$ (as $i\to\infty$) with
corresponding eigenfunctions
\begin{equation}\label{eigenfunction}
w_{mi} (x,y)= \varphi_{m,i}(y)\sin(mx)\, , \ \ m,i\in{\mathbb{N}}\, .
\end{equation}
The eigenfunction $w_{mi}$  has $m$ nodal sets in the $x$ direction while the index $i$ is not related to the number of nodal sets of $w_{mi}$  in
the $y$ direction. The functions $\varphi_{m,i}$ are combinations of hyperbolic and trigonometric sines and cosines, being either even or odd
with respect to $y$.

\begin{definition}[Longitudinal/torsional eigenfunctions]\label{df:longtors}
If $\varphi_{m,i}$ is even we say that the eigenfunction \eqref{eigenfunction} is longitudinal and if $\varphi_{m,i}$ is odd we say that the eigenfunction
\eqref{eigenfunction} is torsional.
\end{definition}

The order between longitudinal and torsional eigenvalues does not follow a simple rule and we will not order
them according to \eqref{eq:eigenvalueH2L2}. We also consider the following buckling problem:
\begin{equation}\label{eq:eigenvalueH2uxintroduction}
\left\{
\begin{array}{rl}
\Delta^2 w + \Lambda w_{xx} = 0& \textrm{in }\Omega\\
w = w_{xx} = 0& \textrm{on }\{0,\pi\}\times[-l,l]\\
w_{yy}+\sigma w_{xx} = 0& \textrm{on }[0,\pi]\times\{-l,l\}\\
w_{yyy}+(2-\sigma)w_{xxy} = 0& \textrm{on }[0,\pi]\times\{-l,l\}\, .
\end{array}
\right.
\end{equation}
We denote the associated eigenvalues by $\Lambda_{m,i}$. We also denote $\Lambda_{1,1}$, the least eigenvalue, by $\Lambda_1$. It is
straightforward that
\begin{empheq}{align}\label{thismeans}
\Delta^2 w_{mi} =-\Lambda_{m,i} (w_{mi} )_{xx}=m^2\Lambda_{m,i} w_{mi} =\lambda_{m,i} w_{mi}\, ,
\end{empheq}
which proves that every eigenfunction $w_{mi}$  of \eqref{eq:eigenvalueH2uxintroduction} is also an eigenfunction
of \eqref{eq:eigenvalueH2L2} and the eigenvalues are related by
\begin{empheq}{align}
\lambda_{m,i} = m^2\Lambda_{m,i}\, ,\qquad\forall\, m,i\in{\mathbb{N}}\, .
\end{empheq}
Therefore, $(w_{mi})$ is a complete orthogonal system of eigenfunctions associated to both the eigenvalue problems \eqref{eq:eigenvalueH2L2}
and \eqref{eq:eigenvalueH2uxintroduction}. In the sequel, we normalize the eigenfunctions so that
\begin{empheq}{align}\label{normltwowmk}
 &\|w_{mi}\|_0^2=1, \quad\| (w_{mi})_x \|_0^2 = m^2, \quad\|w_{mi}\|_2^2=m^2\Lambda_{ m,i}=\lambda_{ m,i}\, .
\end{empheq}

Let us now explain how these eigenfunctions enter in the stability analysis of \eqref{enlm}. We will assume that $\Omega = (0, \pi) \times (-l, l)$, with
\begin{empheq}{align}\label{lsigma}
l=\frac{\pi}{150}\ ,\quad\sigma=0.2\ ,
\end{empheq}
so that the ratio between the longitudinal and transversal lengths is approximately the same as in the original Tacoma Bridge (see \cite{ammann}) and
$\sigma$ is the Poisson ratio of a mixture between concrete and steel. In addition, we will order the eigenvalues $\Lambda_{m,i}$ in an increasing
sequence which will be denoted by $(\Lambda_k)$.

As we have mentioned, there is no simple rule describing the order between longitudinal and torsional eigenvalues. Computations, by the Newton's methods, show that the first 105 eigenvalues of \eqref{eq:eigenvalueH2uxintroduction} are longitudinal and the first 10 are displayed in Table \ref{eigentable} (up to a maximum error of $10^{-2}$).
\begin{table}[h]
\begin{center}
\begin{tabular}{cccccccccc}
\hline
$\Lambda_1$ &$\Lambda_2$ &$\Lambda_3$ &$\Lambda_4$ &$\Lambda_5$ &$\Lambda_6$ &$\Lambda_7$ &$\Lambda_8$ &$\Lambda_9$ &$\Lambda_{10}$\\
\hline
0.96& 3.84& 6.64& 15.36& 24.00 & 34.57& 47.06& 61.48& 77.82& 96.09\\
\hline
\end{tabular}
\caption{First 10 eigenvalues of (\ref{eq:eigenvalueH2uxintroduction}).}\label{eigentable}
\end{center}
\end{table}

Between the 105th eigenvalue and the next longitudinal one, there are at least 10 that are torsional and are listed in Table \ref{eigentable2} (up to a maximum error of $10^{-1}$).

\begin{table}[h]
\begin{center}
\begin{tabular}{cccccccccc}
\hline
$\Lambda_{106}$ &$\Lambda_{107}$ &$\Lambda_{108}$ &$\Lambda_{109}$ &$\Lambda_{110}$ &$\Lambda_{111}$ &$\Lambda_{112}$ &$\Lambda_{113}$ &$\Lambda_{114}$ &$\Lambda_{115}$\\
\hline
10943.6 &  10946.5 &  10951.2 &  10957.8 &  10966.2 &  10976.6 &  10988.8 &
11003.0 &  11019.0 &  11036.9\\
\hline
\end{tabular}
\caption{First 10 torsional eigenvalues of (\ref{eq:eigenvalueH2uxintroduction}).}\label{eigentable2}
\end{center}
\end{table}

The large discrepancies that appear in Tables \ref{eigentable} and \ref{eigentable2} suggest to restrict the attention to the lower eigenvalues. In order to select a reasonable number of low eigenvalues, let us recall what may be seen in actual bridges. A few days prior to the Tacoma Bridge collapse, the project engineer L.R.\ Durkee wrote a letter (see \cite[p.28]{ammann}) describing the oscillations
which were previously observed. He wrote: {\em Altogether, seven different motions have been definitely identified on the main span of the bridge,
and likewise duplicated on the model. These different wave actions consist of motions from the simplest, that of no nodes, to the most complex, that
of seven modes}. Moreover, Farquharson \cite[V-10]{ammann} witnessed the collapse and wrote that {\em the motions, which a moment before had involved a number
of waves (nine or ten) had shifted almost instantly to two}. This means that an instability occurred and changed the motion of the deck from the ninth or
tenth longitudinal mode to the second torsional mode. In fact, Smith-Vincent \cite[p.21]{tac2} state that this shape of torsional oscillations is the only
possible one, see also \cite[Section 1.6]{2015gazzola} for further evidence and more historical facts. Therefore, the relevant eigenvalues corresponding
to oscillations visible in actual bridges, are the longitudinal ones in Table \ref{eigentable} and the torsional one $\Lambda_{107}$ in Table \ref{eigentable2}.
From \cite[p.20]{ammann} we also learn that in the months prior to the collapse {\em the modes of oscillation frequently changed}, which means that some
modes were unstable. In order to study the transition between modes of oscillation from longitudinal to torsional, we complement the instability result given
by Theorem \ref{th:newstability} ii) from Section \ref{section:theorems-stability} with some numerical experiments using the \textit{Scipy library} \cite{scipy}. This analysis is presented in Section \ref{section:numerics} and will be performed with the eigenvalues $\Lambda_3, \Lambda_4, \ldots, \Lambda_{10}, \Lambda_{107}$, which appear enough for a reliable stability analysis. The numerical solutions to our experiments reported in Figure \ref{figure1} precisely show the
instability of some longitudinal modes with a sudden appearance of a torsional oscillation.

\section{Well-posedness and asymptotic behavior} \label{section:weel-posedeness}

Let us first make clear what we mean by solution of \eqref{enlm}.

\begin{definition}[Weak solution]\label{df:weaksolution}
Let $U_0\in H^{2}_*(\Omega)$, $V_0\in L^2(\Omega)$, $F\in \mathcal{C}^0([0,T],L^2(\Omega))$ for $T>0$. A weak solution of \eqref{enlm} is a function
\begin{empheq}{align}
U\in \mathcal{C}^ 0([0,T],H^{2}_*(\Omega))\cap \mathcal{C}^1([0,T],L^2(\Omega))\cap \mathcal{C}^2([0,T],\mathcal{H})
\end{empheq}
such that, $U(0) = U_0$, $U'(0) = V_0$ and
\begin{empheq}{align}\label{weakform}
{\langle U'',V \rangle} + \delta (U',V)_0 + (U,V)_2 +\phi(U) ({U_x, V_x})_0 = (F,V)_0\,,
\end{empheq}
for all $t\in[0,T]$ and all $V\in H^{2}_*(\Omega)$.
\end{definition}

Then we prove existence and uniqueness of a weak solution for \eqref{enlm}.

\begin{theorem}[Existence and uniqueness]\label{exist1}
Given $\delta \in \mathbb{R}$, $S>0$, $P\geq0$, $T>0$, $U_0\in H^{2}_*(\Omega)$, $ V_0\in L^2(\Omega)$ and $F\in \mathcal{C}^0([0,T],L^2(\Omega))$,
there exists a unique weak solution $U$ of \eqref{enlm}. Moreover, it satisfies, for all $t\in[0,T]$,
\begin{empheq}{align}\label{energyid}
\frac12\|{U'}\|_0^2 +  &\frac12\|{U}\|_2^2-\frac P2\|{U_x}\|_0^2+\frac S4\|{U_x}\|_0^4  - \int_0^t ({F ,U'})_0 + \delta\int_0^t \|{U'}\|_0^2  \\
&=\frac12\|{V_0}\|_0^2 +\frac12\|{U_0}\|_2^2 -\frac{P}2\|{(U_0)_ x}\|_0^2 +\frac{S} 4\|{(U_0)_x}\|_0^4\,.
\end{empheq}
\end{theorem}
\vspace{.25cm}

Theorem {\rm{\ref{exist1}}} may also be proved for negative $P$ with no change in the arguments. In elasticity, this situation corresponds to a plate
that has been stretched rather than compressed, which does not occur in actual bridges.
Equation \eqref{energyid} is physically interpreted as an energy balance where the kinetic energy is
$$
{\mathcal K}(U; t)=\frac12 \|{U'(t)}\|_0^2\,,
$$
the elastic potential energy is
$$
{\mathcal P}(U;t)=\frac12\|{U(t)}\|_2^2-\frac P2\|{U_x(t)}\|_0^2+\frac S4\|{U_x(t)}\|_0^4\,,
$$
the exterior exchange is
$$
{\mathcal F} (U; t)= - \int_0^t (F, U')_0\,,
$$
the frictional loss is
$$
\mathcal{W}(U; t) = \delta \int_0^t \| U'\|_0^2 \,
$$
and the law conservation \eqref{energyid} says that
\[
{\mathcal K}(U; t) + {\mathcal P}(U;t)+ {\mathcal F} (U; t) + \mathcal{W}(U; t)  \quad \text{is constant for all } t \geq0.
\]
In turn, the elastic energy consists in the bending energy $\|{U(t)}\|_2^2/2$ and the stretching energy $-P\|{U_x(t)}\|_0^2/2+S\|{U_x(t)}\|_0^4/4$.
The total mechanical energy ${\mathcal E}(U;t) = {\mathcal K}(U; t) + {\mathcal P}(U;t)$ is the sum of the kinetic and potential energies so that ${\mathcal E}(U;0)$ is the initial energy. Moreover, we see from \eqref{energyid} that
\begin{empheq}{align}\label{physicsenergyid}
{\mathcal E}(U;t)={\mathcal E}(U;0)-\delta\int_0^t \|{U'}\|_0^2+\int_0^t ({F ,U'})_0  \, ,\qquad\forall\, t \in [0,T]\,.
\end{empheq}
In the isolated case with no damping and no load, i.e.\ with $\delta = 0$ and $F= 0$, the mechanical energy is constant. In the unforced case $F=0$,
\eqref{physicsenergyid} shows that the energy is monotonic according to the sign of $\delta$.

\smallbreak
Next we analyze the asymptotic behavior of the solution, as $t \to +\infty$, under the influence of a positive damping ($\delta>0$), when $F$ is
time independent. As we will see, the solution's behaviour is also influenced by the properties of the
stationary problem associated to \eqref{enlm}, namely
\begin{empheq}{align}\label{stationary}
\left\{
\begin{array}{rl}
\Delta^2 U - \phi(U) U_{xx} = F& \textrm{in }\Omega\\
U = U_{xx} = 0& \textrm{on }\{0,\pi\}\times[-l,l]\\
U_{yy}+\sigma U_{xx} = 0& \textrm{on }[0,\pi]\times\{-l,l\} \\
U_{yyy}+(2-\sigma)U_{xxy} = 0& \textrm{on }[0,\pi]\times\{-l,l\}\, .
\end{array}
\right.
\end{empheq}

When the prestressing $P$ is not larger than the least eigenvalue $\Lambda_1$, the energy ${\mathcal P}_{F}$ is convex and problem \eqref{stationary}
has a unique solution; see \cite{2014al-gwaizNATMA,2015ferreroDCDSA}. In this case we have the following result.

\begin{theorem}[Behaviour at $\infty$ with small $P$]\label{thmcoercive}
 Let $\delta>0$, $S>0$,  $U_0\in H^{2}_*(\Omega)$, $V_0\in L^2(\Omega)$ and $F \in L^2(\Omega)$. If $0\le P\le \Lambda_1$,
 then the solution $U$ of \eqref{enlm} is such that $U(t) \rightarrow \overline U$ in $H^{2}_*(\Omega)$, $U'(t) \to 0$ in $L^2(\Omega)$ as $t \rightarrow \infty$, where $\overline U \in H^{2}_*(\Omega)$ is the unique solution of \eqref{stationary}.
 \end{theorem}

Next, we consider the case with absence of the load $F$. In this case, \eqref{physicsenergyid} tells us that the solution moves towards lower energy levels
whenever $\delta>0$. Theorem \ref{oculos} shows that the solution may exhibit different behaviours as $P$ crosses $\Lambda_1$, namely for
$\Lambda_1<P\le\Lambda_2$, where $\Lambda_2$ denotes the second eigenvalue of problem \eqref{eq:eigenvalueH2uxintroduction}. In this range of the
parameter $P$, the eigenfunctions of problem \eqref{eq:eigenvalueH2uxintroduction} come into play. We recall that $U = 0$ and
$U = \pm \lambda_+{}w_1{}$, with $ \lambda_+ =\sqrt{\frac{P- \Lambda_1}{S}}$, are all the stationary solutions of \eqref{enlm};
see \cite[Theorem 7]{2014al-gwaizNATMA}.

\begin{theorem}[Behaviour at $\infty$ with $P$ not so small and negative energy]\label{oculos}
Let $\delta>0$, $S>0$,
 $U_0\in H^{2}_*(\Omega)$, $V_0\in L^2(\Omega)$, $\Lambda_1<P\le\Lambda_2$
and $U$ be the solution of \eqref{enlm} with $F\equiv0$.
If ${\mathcal E}(U; 0) <0$, then $U(t) \to\overline U$ in $H^{2}_*(\Omega)$, $U'(t) \to 0$ in $L^2(\Omega)$ as $t\to\infty$, where
$$
\overline U =\left\{\begin{array}{ll}
\lambda_+ {}w_1{}\quad & \textrm{if} \quad ({U_0,{}w_1{}})_2>0\\
-\lambda_+ {}w_1{} \quad & \textrm{if} \quad ({U_0,{}w_1{}})_2<0\,.
\end{array}\right.$$
\end{theorem}

\begin{remark}
The open set $\mathcal{N}:= \{ U_0\in H^{2}_*(\Omega); \, {\mathcal E}(U_0,0)< 0\}  \subset H^{2}_*(\Omega)$ consists of two path-connected components, one contains $\lambda_+{}w_1{}$, the other contains $-\lambda_+{}w_1{}$ and $\mathcal{N} \cap \{{}w_1{}\}^{\perp} = \emptyset$. Moreover, the origin is the only point in the intersection of the boundary of these components.
\end{remark}

The next result describes the invariance of the solution according to initial data. This turns out to be important, in particular, to prove our stability/instability results of simple modes.

\begin{theorem}[Invariance according to data]\label{thm-vkperp}
Let $w$ be an eigenfunction of \eqref{eq:eigenvalueH2uxintroduction}. If $F \in [w]^{\perp} \subset L^2(\Omega)$, $U_0 \in [w]^{\perp} \subset H^{2}_*(\Omega)$, $V_0 \in [w]^{\perp} \subset L^2(\Omega)$ and  $U$ is the weak solution of \eqref{enlm}, then $U(t) \in [w]^{\perp} \subset H^{2}_*(\Omega)$ for all $t$.
\end{theorem}

Theorem \ref{thm-vkperp} says that if the initial data and forcing term have null coordinates in some entries of their Fourier series,
then the solution will also have the corresponding coordinates null.

As a first consequence, we have a convergence result for initial data in $[w_1]^{\perp}$, with positive energy, in the same range of parameters from Theorem \ref{oculos}.
\begin{corollary}[Behaviour at $\infty$ with $P$ not so small and positive energy]\label{th:perpv1}
Let $\delta>0$, $S>0$,  $U_0\in [w_1]^{\perp} \subset H^{2}_*(\Omega)$, $V_0\in [w_1]^{\perp} \subset L^2(\Omega)$, $\Lambda_1<P\le\Lambda_2$ and $U$ be the solution of \eqref{enlm} with $F\equiv0$. Then $U(t) \in [w_1]^{\perp} \subset H^{2}_*(\Omega)$ for all $t \geq 0$ and $U(t) \to 0$ in $H^{2}_*(\Omega)$, $U'(t) \to 0$ in $L^2(\Omega)$ as $t\to\infty$.
\end{corollary}

Similar results are also available for $P>\Lambda_2$. However, the physically meaningful values of prestressing are $P\le\Lambda_2$ since otherwise
the equilibrium positions of the plate may take unreasonable shapes such as ``multiple buckling''.

\section{Stability of the simple modes}\label{section:theorems-stability}

In this section we consider the case where the  problem is isolated, i.e. with no damping and no load. From Theorem \ref{thm-vkperp} we know that
if the initial data ${U}_0, V_0$ have only one nontrivial component in their Fourier expansions, that is,
\begin{empheq}{align}{U}_0= u_{0,m,i}w_{mi}, \quad V_0 = v_{0,m,i}w_{mi},\end{empheq}
for some $m,\, i$, then the solution ${U}$ has the same property and can be written as
\begin{empheq}{align}\label{simplemodeintro}
{U}(x,y,t) =\varphi(t)w_{mi}(x,y)
\end{empheq}
for some $\varphi\in C^2([0,\infty), {\mathbb{R}})$ satisfying $\varphi(0)=u_{0,m,i}$ and $\varphi'(0)=v_{0,m,i}$.
We call \eqref{simplemodeintro} a $(m,i)$-simple mode of oscillation of \eqref{enlm} and the function $\varphi$ is called the coordinate of the
$(m,i)$-simple  mode. One may be skeptic on the possibility of seeing a simple mode on the deck of a bridge; however,
from \cite[p.20]{ammann} we learn that in the months prior to the collapse {\em one principal mode of oscillation prevailed} and that {\em the modes
of oscillation frequently changed}, which means that the motions were ``almost simple modes'' and that some of them were unstable.
We are so led to consider initial data ${U}_0, V_0$ with two nontrivial components in their Fourier expansions. In this case, we have

\begin{proposition}\label{twocomponents}
Assume that $F\equiv 0$, $\delta=0$, and
\begin{empheq}{align}{U}_0= u_{0,m,i}w_{mi}+u_{0,n,k}w_{nk}, \quad V_0 = v_{0,m,i}w_{mi}+v_{0,n,k}w_{nk},\end{empheq}
for some $m,\, n,\, i,\, k\in\mathbb{N}$ with $(m,i)\neq(n,k)$ and some $u_{0,m,i},\, u_{0,n,k},\, v_{0,m,i},\, v_{0,n,k}\in\mathbb{R}$.
Then the solution ${U}$ of \eqref{enlm} can be written as
\begin{empheq}{align}\label{doublemodeintro}
{U}(x,y,t) =\varphi(t)w_{mi}(x,y)+\psi(t)w_{nk}(x,y)\, ,
\end{empheq}
where $\varphi$ and $\psi$ belong to $C^2([0,\infty), {\mathbb{R}})$ and satisfy the following nonlinear system of ODE's:
\begin{equation}\label{systemODE}
\left\{\begin{array}{ll}
\!\!\!\varphi''(t)+m^2(\Lambda_{m,i}\!-\!P)\varphi(t)+Sm^2\big[m^2\varphi(t)^2+n^2\psi(t)^2\big]\varphi(t)=0\vspace{5pt}\\
\!\!\!\psi''(t)+n^2(\Lambda_{n,k}\!-\!P)\psi(t)+Sn^2\big[m^2\varphi(t)^2+n^2\psi(t)^2\big]\psi(t)=0
\end{array}\right.
\end{equation}
with initial data
$$\varphi(0)=u_{0,m,i},\quad \psi(0)=u_{0,n,k},\quad\varphi'(0)=v_{0,m,i},\quad\psi'(0)=v_{0,n,k}\, .$$
\end{proposition}

The proof of Proposition \ref{twocomponents} may be obtained by replacing \eqref{doublemodeintro} into \eqref{enlm}, by multiplying the so obtained
equation with $w_{mi}$ and $w_{nk}$, by integrating over $\Omega$, and by using \eqref{thismeans}. One would like to know whether the following
implication holds:
\begin{equation}\label{lyapunov}
|u_{0,n,k}|+|v_{0,n,k}|\ll|u_{0,m,i}|+|v_{0,m,i}|\ \Longrightarrow\ \|\psi\|_\infty\ll\|\varphi\|_\infty\, .
\end{equation}
If this happens, we say that $\varphi$ is stable with respect to $\psi$, otherwise we say that it is unstable.
Hence, the test of stability consists in studying the stability of the system \eqref{systemODE}. Let us make all this more precise, especially because
\eqref{lyapunov} may be difficult to check.\par
The system \eqref{systemODE} is isolated, its energy ${\mathcal E}={\mathcal E}(u_{0,m,i},u_{0,n,k},v_{0,m,i},v_{0,n,k})$ is constant, and it is given by
\begin{equation}\label{energy}
\begin{array}{rcl}
{\mathcal E}\!&=&\!\tfrac{\varphi'^2}{2}\!+\!\tfrac{\psi'^2}{2}\!
+\!m^2(\Lambda_{m,i}\!-\!P)\tfrac{\varphi^2}{2}\!+\!n^2(\Lambda_{n,k}\!-\!P)\tfrac{\psi^2}{2}\!+\!
S\tfrac{(m^2\varphi^2+n^2\psi^2)^2}{4}\vspace{5pt}\\
\!&\equiv&\!\tfrac{v_{0,m,i}^2}{2}\!+\!\tfrac{v_{0,n,k}^2}{2}\!+\!m^2(\Lambda_{m,i}\!-\!P)\tfrac{u_{0,m,i}^2}{2}\!
+\!n^2(\Lambda_{n,k}\!-\!P)\tfrac{u_{0,n,k}^2}{2}+S\tfrac{(m^2u_{0,m,i}^2+n^2u_{0,n,k}^2)^2}{4}\, =:E_0.
\end{array}\end{equation}

Let us make precise what we intend for stability of modes for the isolated problem.

\begin{definition}[Stability]\label{defstab*}
The $(m,i)$-simple  mode $\varphi(t)w_{mi}(x,y)$ is said to be $(n,k)$ linearly stable if $\xi\equiv0$
is a stable solution of the linear Hill equation
\begin{equation}\label{hill}
\xi''+a(t)\xi=0\, ,\qquad a(t)=n^2(\Lambda_{n,k}-P)+Sm^2n^2\varphi(t)^2\quad \forall \, t\, .
\end{equation}
\end{definition}

Since \eqref{hill} is linear, this is equivalent to state that all the solutions of \eqref{hill} are bounded. Since \eqref{systemODE}
is nonlinear, the stability of $\varphi w_{mi}$  depends on the initial conditions and, therefore, on the corresponding energy \eqref{energy}.
On the contrary, the linear instability of $\varphi w_{mi}$  occurs when the trivial solution of \eqref{hill} is unstable: in this case,
if the initial energy is almost all concentrated in $(\varphi(0),\varphi'(0))$, the component $\varphi$ conveys part of its energy to $\psi$ for some $t>0$.

\begin{remark}
The condition \eqref{lyapunov} is usually called Lyapunov stability which is much stronger than the linear stability as characterized by
Definition \ref{defstab*}. In some cases closely related to our problem, these two definitions coincide, see \cite{ghg}. We also point out
that the equation \eqref{hill} may be replaced by its nonlinear counterpart $\xi''+a(t)\xi+Sn^4\xi^3=0$ without altering Definition
\ref{defstab*}, see \cite{ortega}.
\end{remark}

Let $j \in {\mathbb{N}}$ (including 0) and set
\begin{empheq}{align}I_j = \Big( j(2j+1), (j+1)(2j+1) \Big)\, , \quad K_j =\Big( (j+1)(2j+1), (j+1)(2j+3)\Big).\end{empheq}

These intervals, that were found by Cazenave-Weissler \cite{1996cazenaveQAM}, govern the stability of the modes for large energies.
In fact, the following statement holds.

\begin{theorem}[Stability/Instability of simple modes with large energy]\label{th:newstability}
Let $F\equiv 0$, $\delta = 0$, $S>0$, $P< \min\{\Lambda_{m,i},  \Lambda_{n,k}\}$, $m,n,i,k\in {\mathbb{N}}$, and set  $\gamma = \frac{n^2}{m^2}$.
\begin{enumerate}[i)]
  \item  If $\gamma \in I_j$ for some $j\in{\mathbb{N}} \cup\{0\}$, then any $(m,i)$-simple  mode with sufficiently large energy $E_0$ is $(n,k)$ linearly stable.

 \item  If $\gamma \in K_j$ for some $j\in{\mathbb{N}} \cup\{0\}$, then any $(m,i)$-simple  mode with sufficiently large energy $E_0$ is $(n,k)$ linearly unstable.
\end{enumerate}
\end{theorem}

\begin{remark}
As a consequence, we infer that any $(m,i)$-simple  mode with sufficiently large energy is $(n,k)$ linearly unstable for some $n,k\in{\mathbb{N}}$. Indeed, given $m,i\in{\mathbb{N}}$
we can choose $n =3m$ so that $\gamma\in K_1$.
\end{remark}

Next, we present a stability result in the case of small energy.

\begin{theorem}[Stability of simple modes with small energy]\label{stable}
Let $F\equiv 0$, $\delta = 0$, $S>0$, $P< \min\{\Lambda_{m,i},  \Lambda_{n,k}\}$, $m,n,i,k\in {\mathbb{N}}$. If
\begin{equation}\label{nm}
\frac{n}{m}\sqrt{\frac{\Lambda_{n,k}-P}{\Lambda_{m,i}-P}}\not\in\mathbb{N}^+\qquad\mbox{or}\qquad
\frac{n}{m}\sqrt{\frac{\Lambda_{n,k}-P}{\Lambda_{m,i}-P}}\in\mathbb{N}^+\quad\mbox{and}\quad 4n^2<3m^2\,,
\end{equation}
then a $(m,i)$-simple mode with sufficiently small energy $E_0$ is $(n,k)$ linearly stable.
\end{theorem}

Clearly, \eqref{nm} occur with probability 1 among all possible random choices of the parameters involved. Moreover, as we shall see in next section,
it is certainly satisfied in all the problems of physical interest.

\section{Some numerics showing the instability of modes}\label{section:numerics}

In this section we assume that $\Omega = (0, \pi) \times (-l, l)$, with \eqref{lsigma}, so that the ratio between the longitudinal and transversal lengths is approximately the same as in the original Tacoma Bridge. We also take $P = \frac{\Lambda_{1,1}}{2} = 0.48$ and $S=3$ and we will complement the analysis from last section with some numerical experiments.

The second torsional eigenvalue of \eqref{eq:eigenvalueH2uxintroduction}, namely $\Lambda_{107}$ from Table \ref{eigentable2}, is of special interest because it corresponds to the torsional mode observed just before the Tacoma Bridge collapse. Using Theorem \ref{th:newstability} and some numerical results we indicate how instability occurs and how an almost purely longitudinal oscillation can suddenly start oscillating in a torsional fashion. For that we will perturb longitudinal simple modes associated to the eigenvalue $\Lambda_3, \Lambda_4, \ldots, \Lambda_{10}$ from Table \ref{eigentable} by a torsional simple mode associated to the second torsional eigenvalue $\Lambda_{107}$ from Table \ref{eigentable2}. This analysis is summarized as:
\begin{enumerate}[i)]
\item If the energy is sufficiently small, then Theorem \ref{stable} guarantees stability for these longitudinal simple modes under perturbation by second torsional simples modes; note that $4n^2<3m^2$ holds since $n=2$ and $m=3, 4, \ldots, 10$.
\item If the energy is sufficiently large, then Theorem \ref{th:newstability} i) guarantees stability for these longitudinal simple modes under perturbation by second torsional simples modes. Observe that $\gamma = 4/m^2<1$ ($m =3, 4, \ldots, 10$) and so $\gamma \in I_0$.
\item Hence, if instability occurs, this necessarily happens for some intermediate value of energy. Indeed, we have observed such phenomena in some numerical experiments and the range of energy and corresponding initial data  are presented in Table \ref{results1}. This table considers system \eqref{systemwz}-\eqref{hill-nova} below, which is equivalent to system \eqref{systemODE}, with initial data
\begin{equation}\label{initialdatapert}
\varphi(0) = u_0, \ \varphi'(0)= 0, \ \psi(0) = \frac{u_0}{1000}, \ \psi'(0)=0,
\end{equation}
and the value $\psi(0) = \frac{u_0}{1000}$ is intended to represent a small perturbation by the second torsional simple mode ($\varphi$ corresponds to longitudinal and $\psi$ to torsional). The shooting interval corresponds to the range of the parameter $u_0$ while the energy interval is the corresponding range for the energy associated to \eqref{systemwz}-\eqref{hill-nova}, which is given by \eqref{energy} divided by $S$.
\end{enumerate}

\begin{table}[h!]
\centering
\begin{tabular}{ccccc}
& $\Lambda_3$ & $\Lambda_4$& $\Lambda_5$& $\Lambda_6$ \\ \hline \hline
\begin{tabular}{@{}c@{}}shooting \\ interval\end{tabular}
&(50, 64)&(37, 41)&(29, 30.83)&(23.7, 24.87) \\ \hline
\begin{tabular}{@{}c@{}}energy \\ interval\end{tabular}
&(1.57*10$^6$, 4.20*10$^6$)&(473634, 712694)
&(181766, 231447)&(83660.9, 100912) \\ \hline \\
& $\Lambda_7$ & $\Lambda_8$& $\Lambda_9$& $\Lambda_{10}$ \\ \hline \hline
\begin{tabular}{@{}c@{}}shooting \\ interval\end{tabular}&(20.1, 20.66)&(17.15, 17.521)&(14.74, 14, 98)&(12.63, 12.79) \\ \hline
\begin{tabular}{@{}c@{}}energy \\ interval\end{tabular}
&(45510.7, 50517.7)&(26112.2,  28241.4)&(16002.2, 16927.7)
&(10174.3, 10600) \\ \hline
\end{tabular}
\caption{Intervals of instability: longitudinal perturbed by torsional}
\label{results1}
\end{table}

Table \ref{results1} deserves several comments. As already recalled at the end of Section \ref{section:eigenvalue}, the following facts were observed
at the Tacoma Bridge:\par
$\bullet$ prior to the day of the collapse, the deck was seen to oscillate only on the longitudinal modes from the first to the seventh;\par
$\bullet$ the day of the collapse, the deck was oscillating on the ninth or tenth longitudinal mode;\par
$\bullet$ all the oscillations were unstable since the modes of oscillations frequently changed.\par
Moreover, according to Eldridge \cite[V-3]{ammann} (another witness), on the day of the collapse {\em the bridge appeared to be behaving in the
customary manner} and the motions {\em were considerably less than had occurred many times before}. Table \ref{results1} explains why torsional
oscillations did not appear earlier at the Tacoma Bridge even in presence of wider longitudinal oscillations: the critical threshold of amplitude
(the lower bound of the shooting interval) of the longitudinal modes up to the seventh are larger than the thresholds of the ninth and tenth modes.
Although our model and results do not take into account all the mechanical parameters nor yield quantitative measures, we believe that, at least
qualitatively, they give an idea why the Tacoma Bridge collapsed when the longitudinal oscillations displayed {\em nine or ten waves}.

\begin{figure}[h!]
\includegraphics[width=1\textwidth, height=1.5in]{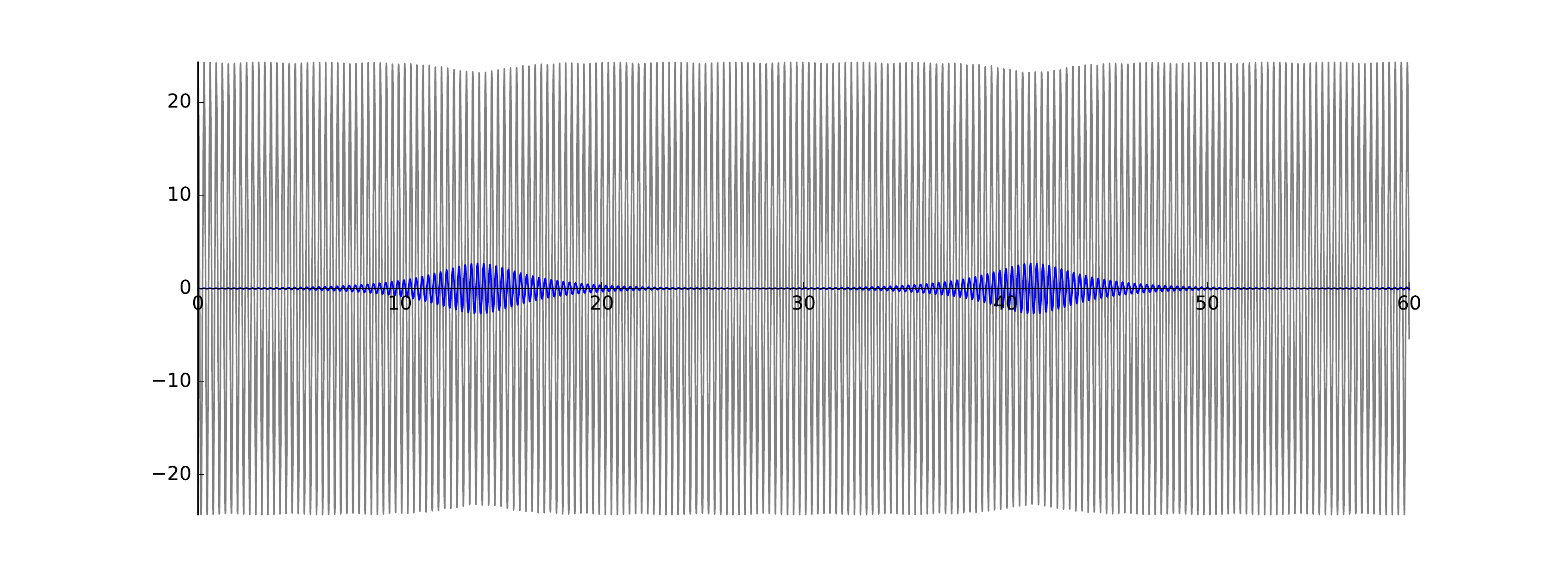}\\
\,\hfill (a) $m=6$, $a=24.3$, $T = 60$\hfill{}\, \\
\includegraphics[width=1\textwidth, height=1.5in]{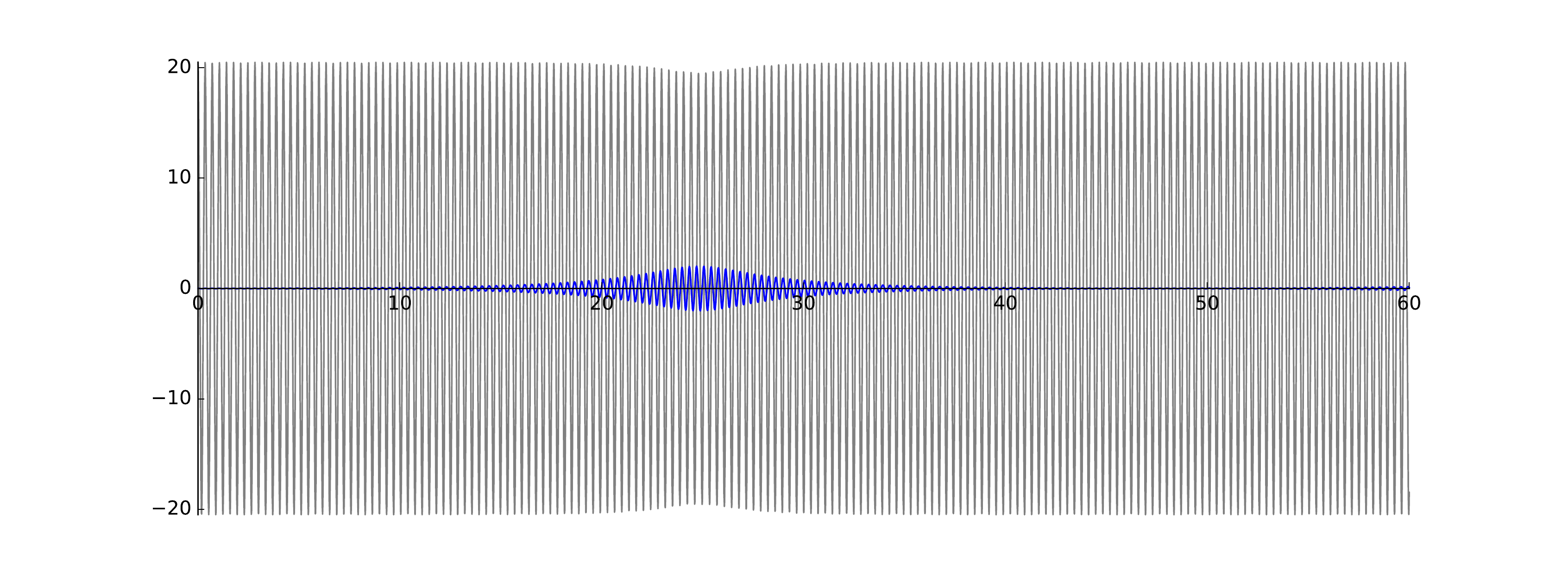}\\
\,\hfill (b) $m=7$, $a=20.5$, $T = 60$\hfill{}\, \\
\includegraphics[width=1\textwidth, height=1.5in]{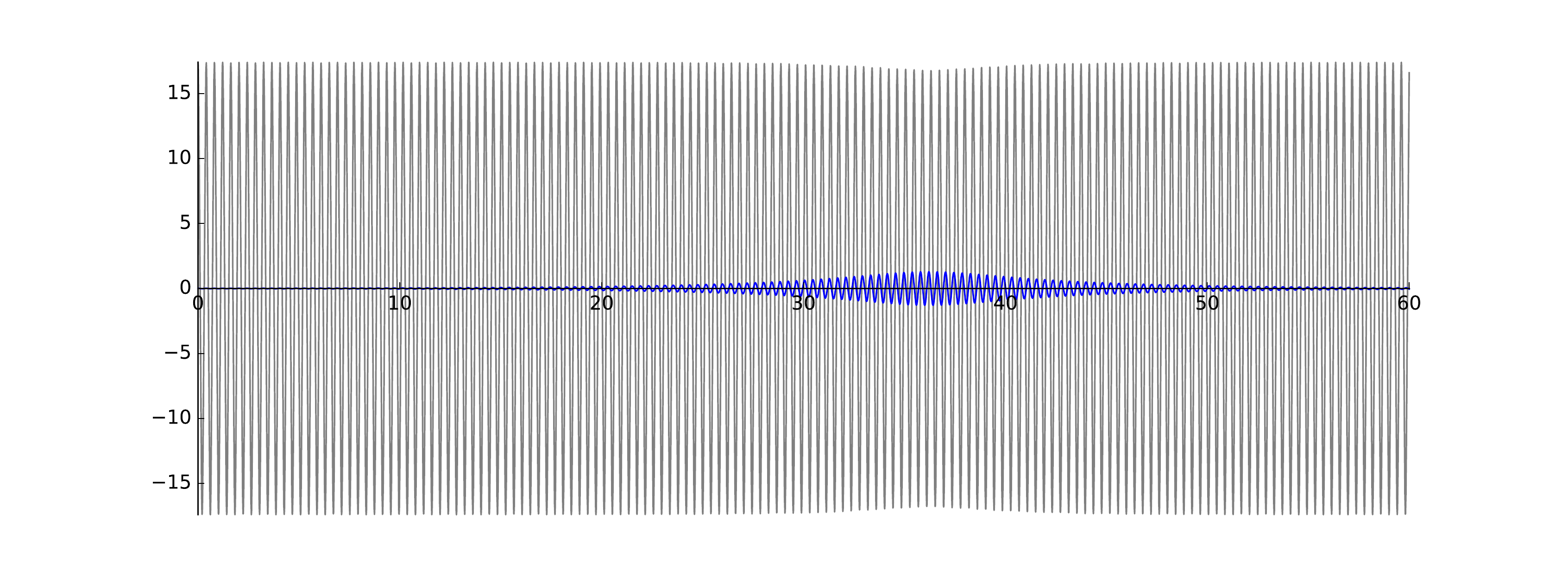}\\
\,\hfill (c) $m=8$, $a=17.42$, $T = 60$\hfill{}\, \\
\includegraphics[width=1\textwidth, height=1.5in]{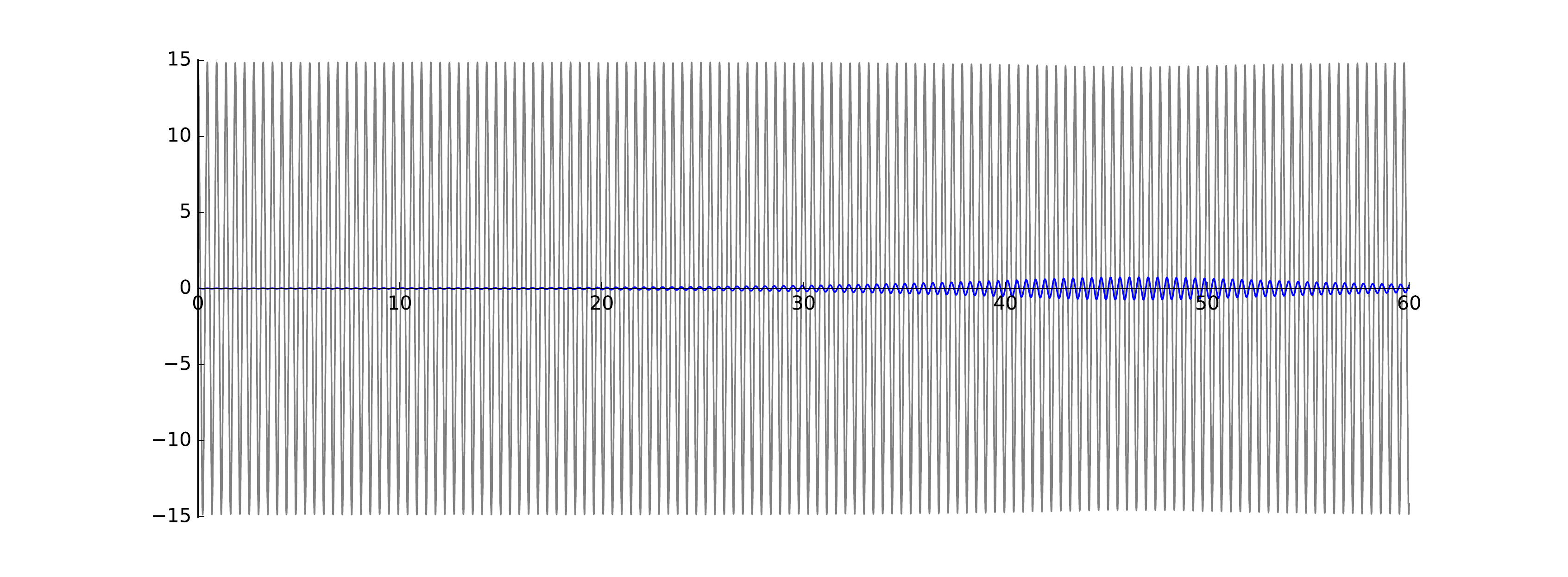}\\
\,\hfill (d) $m=9$, $a=14.89$, $T = 60$\hfill{}\, \\
\caption{Solution displaying instability for $m$ from $6$ through $9$, the amplitude is within the intervals of instability from Table 3.
Gray=longitudinal, Blue=torsional.}
\label{figure1}
\end{figure}

It is well-known by engineers (and also observed in our numerical simulations) that an increment of the damping parameter $\delta$ prevents the appearance of instability. However, such increment is very costly when building a bridge and o good compromise between stiffness and price is of vital importance. So, it is essential to know the optimal damping that guarantees stability. Table \ref{results2} brings the minimum value of $\delta$ that rules out the instability observed in the intervals of Table \ref{results1}; take also into account that $\overline{U} =0$ in Theorem \ref{thmcoercive}.

\begin{table}[h!]
\centering
\begin{tabular}{ccccccccc}
& $\Lambda_3$ & $\Lambda_4$& $\Lambda_5$& $\Lambda_6$ & $\Lambda_7$ & $\Lambda_8$& $\Lambda_9$& $\Lambda_{10}$ \\ \hline \hline
\begin{tabular}{@{}c@{}}energy \\ level\end{tabular}
&2.65*10$^6$&584019&207793&90734.7
&48143.5&26957.6&16459.9&$10358.9$  \\ \hline
\begin{tabular}{@{}c@{}}damping \\ threshold\end{tabular}
&0.48&0.10& 0.03& 0.008& 0.0053& 0.0018& 0.0011&
   0.00046\\ \hline
\end{tabular}
\caption{Damping threshold to rule out instability: longitudinal perturbed by torsional.}
\label{results2}
\end{table}

It appears evident that the damping parameter necessary to rule out high modes may be fairly small, if compared to low modes. This means that
with little economical effort a small damper would have prevented the Tacoma collapse.

\begin{figure}
\includegraphics[width=1\textwidth, height=1.5in]{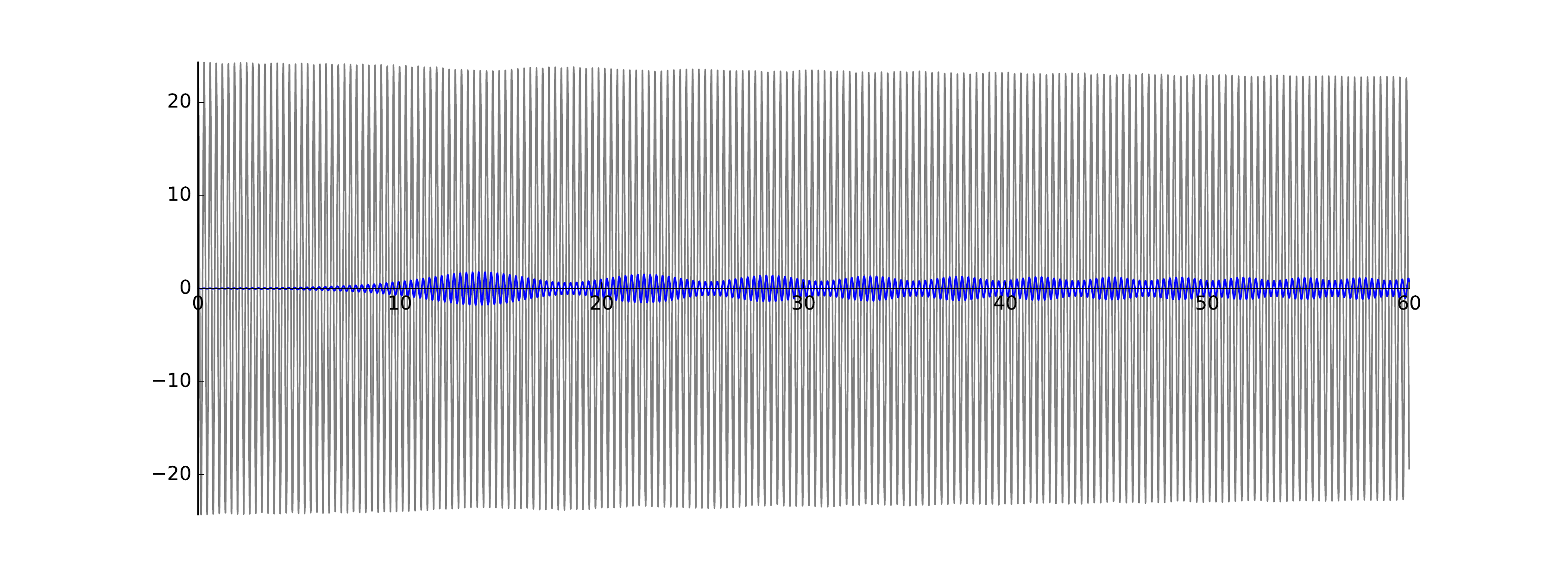}\\
\,\hfill (a) $m=6$, $a=24.3$, $T = 60$, $\delta = 0.03$\hfill{}\, \\
\includegraphics[width=1\textwidth, height=1.5in]{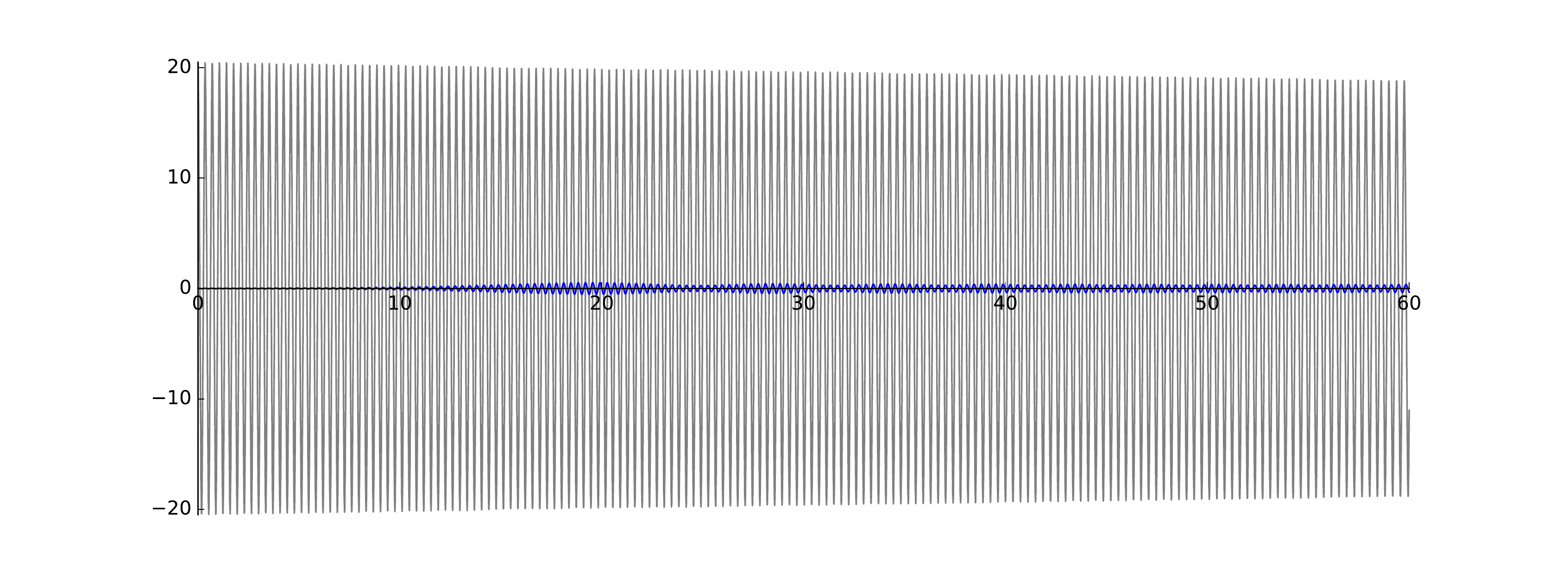}\\
\,\hfill (b) $m=7$, $a=20.5$, $T = 60$, $\delta = 0.004$\hfill{}\, \\
\includegraphics[width=1\textwidth, height=1.5in]{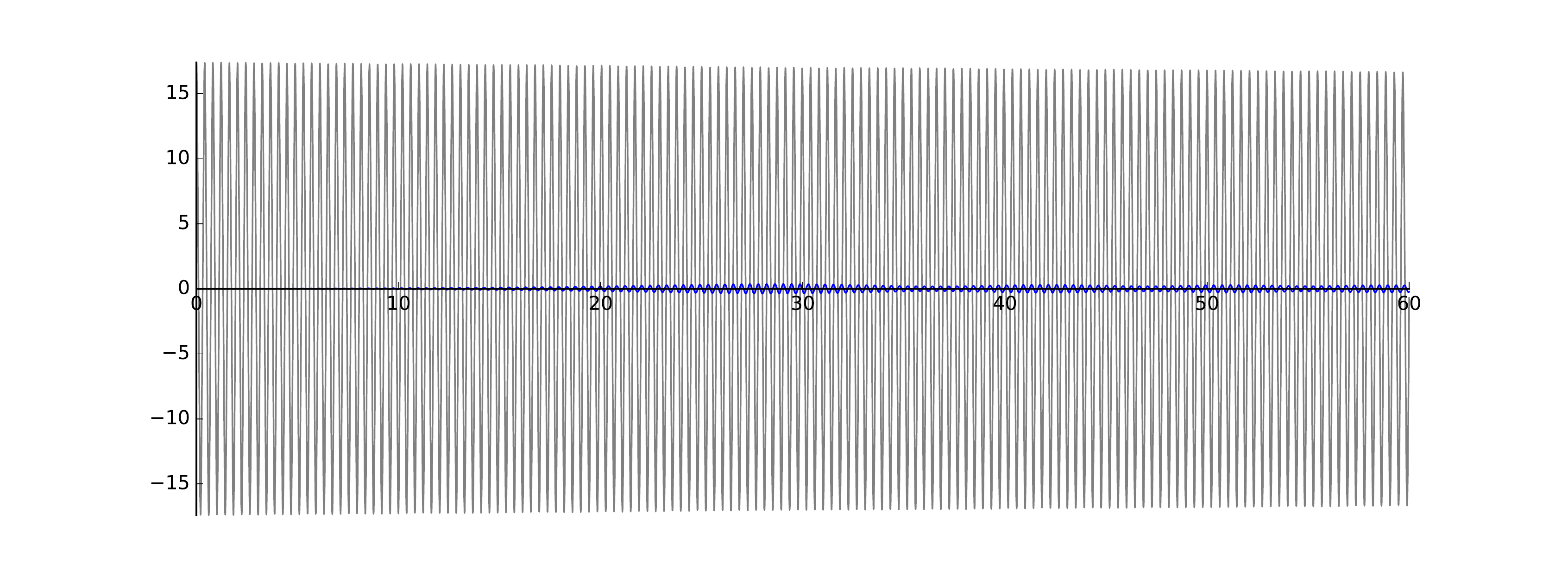}\\
\,\hfill (c) $m=8$, $a=17.42$, $T = 60$, $\delta = 0.002$\hfill{}\, \\
\includegraphics[width=1\textwidth, height=1.5in]{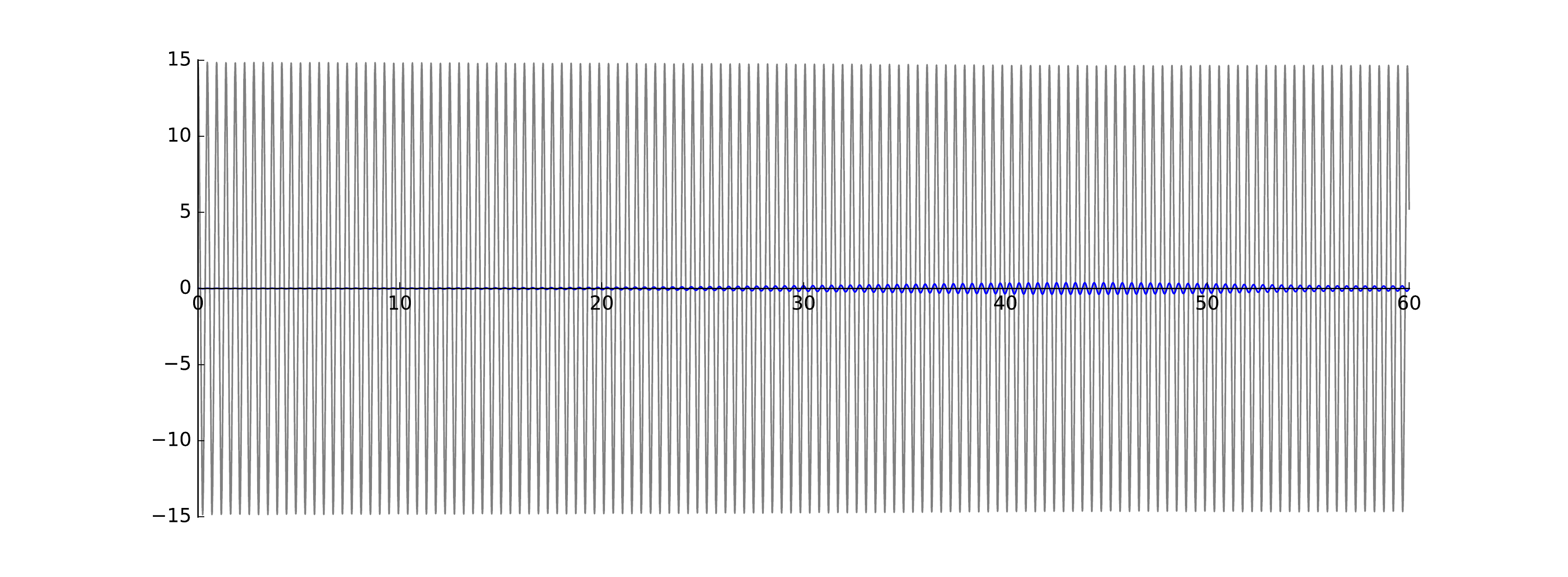}\\
\,\hfill (d) $m=9$, $a=14.89$, $T = 60$, $\delta = 0.0006$\hfill{}\, \\
\includegraphics[width=1\textwidth, height=1.5in]{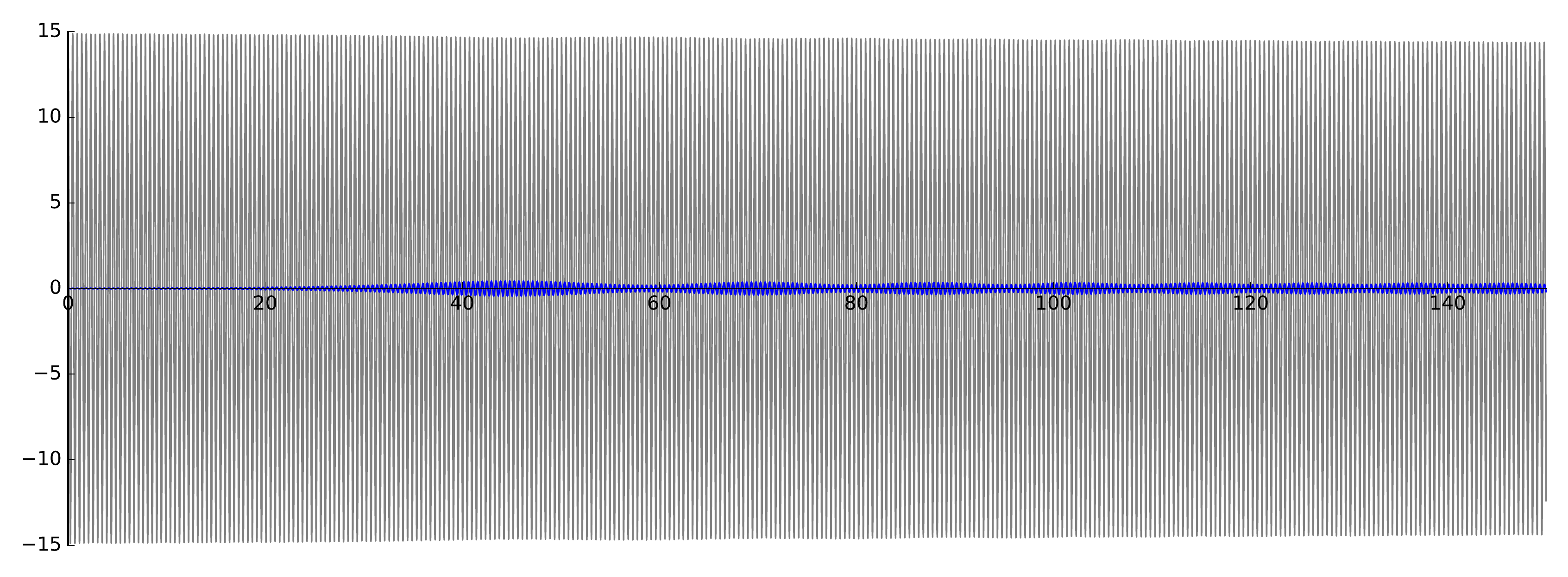}\\
(e) Same plot as (d) over the interval $[0,150]$
\caption{Solution with the same initial data as Figure \ref{figure1} under some damping $\delta$.}
\end{figure}

Even if it falls slightly outside of the range of applications, for the sake of completeness we now do the reverse. Namely, we treat perturbations of the second torsional simple mode by a longitudinal simple mode associated to any of the eigenvalue $\Lambda_3, \Lambda_4, \ldots, \Lambda_{10}$. Again, we will consider initial data as in \eqref{initialdatapert}, but now $\varphi$ corresponds to torsional and $\psi$ to longitudinal. As a direct consequence of Theorem \ref{th:newstability} ii), if the energy is sufficiently large then we get instability for perturbations associated to the eigenvalues $\Lambda_3, \Lambda_5, \Lambda_6, \Lambda_8, \Lambda_9$. Observe that in these situations $\gamma = 3^2/4 \in K_0$, $\gamma = 5^2/4 \in K_1$, $\gamma = 6^2/4 \in K_1$, $\gamma = 8^2/4 \in K_2$ and $\gamma = 9^2/4 \in K_2$, respectively. Table \ref{leastenergytable} shows the energy level and the corresponding initial amplitude, i.e. $u_0$, above which instability appears.

\begin{table}[h!]
\centering
\begin{tabular}{rccccc}

&$ \Lambda_3$&$\Lambda_5$&$\Lambda_6$&$\Lambda_8$&$\Lambda_9$ \\ \hline \hline
Energy &3.82*10$^6$& 205848& 2.04*10$^7$ & 1.19*10$^9$ &4.23*10$^7$\\\hline
Amplitude
&62.5 &30.0 &95.0 &262.4&114.0 \\ \hline

\end{tabular}
\caption{Threshold for instability: torsional perturbed by longitudinal}\label{leastenergytable}
\end{table}

Next we consider the perturbation described in the above paragraph and we indicate, in Table \ref{other-intervals-table}, some intervals where we have detected instability. The interesting fact is that even modes which are stable for large energy (the cases of  $\Lambda_4, \Lambda_7, \Lambda_{10}$ guaranteed by Theorem \ref{th:newstability} i)\,) may experiment intervals of instability. Observe that in these situations we have $\gamma = 4^2/4 \in I_1$, $\gamma = 7^2/4 \in I_2$, $\gamma = 10^2/4 \in I_3$, respectively.

\begin{table}[h!]
\begin{tabular}{cccccccc}
 & $\Lambda_4$& $\Lambda_5$& $\Lambda_6$ & $\Lambda_7$ & $\Lambda_8$& $\Lambda_9$& $\Lambda_{10}$ \\ \hline \hline
\begin{tabular}{@{}c@{}}Shooting\\ interval \end{tabular}  &(40.1, 121)&(30.2, 63.5)&(24.3, 46.2)&
\begin{tabular}{@{}c@{}}(21, 37),\\(65, 141) \end{tabular}&
\begin{tabular}{@{}c@{}}(17, 30), \\ (53, 86) \end{tabular}&
\begin{tabular}{@{}c@{}}(14.5, 26),\\ (43, 65) \end{tabular}&
\begin{tabular}{@{}c@{}}(14, 22),(37, 53),\\ (81.5, 157) \end{tabular} \\ \hline
\end{tabular}
\caption{Intervals where instability is present: torsional perturbed by longitudinal}\label{other-intervals-table}
\end{table}

To conclude, we believe that the numerical experiments of this sections along with the theoretical results from Section \ref{section:theorems-stability} might contribute towards the understanding of bridges oscillation phenomena and indicate further directions of research. Of course, the next step should be to
obtain precise quantitative results.

\section{Proof of existence and uniqueness} \label{section:proofEU}

Here we present the proof for Theorem \ref{exist1}, which is split in several steps. For the existence result we use the Galerkin method, whereas for the uniqueness we argue as in \cite[Section 7.2]{evans}.

\vspace{10pt}
\noindent
{\em Step 1. } Approximating solutions.

In order to build the approximating solutions, we consider the decomposition of $H^{2}_*(\Omega)$ induced by the eigenvalue
problems \eqref{eq:eigenvalueH2L2} and \eqref{eq:eigenvalueH2uxintroduction}. To simplify notation, in this section we will drop the double indexation. The eigenvalues $(\lambda_{m,i})$ will be reorganized in a nondecreasing
sequence $(\lambda_k)$, repeated according to their multiplicity, and the respective eigenfunctions, denoted simply by $(w_k)$, form an orthogonal basis in
$H^{2}_*(\Omega)$ and $L^2(\Omega)$. According to \eqref{normltwowmk}, we normalize the eigenfunctions so that $\|w_k\|_0=1$ and hence $\left( \frac{w_k}{\sqrt{\lambda_k}} \right)$ is an orthonormal basis in $H^{2}_*(\Omega)$.

For all integer $k\ge1$, we set $E_k=\textrm{span}(w_1,\dots,w_k)$ and we consider the orthogonal projections $Q_k:H^{2}_*(\Omega)\to E_k$. We set up the
weak formulation \eqref{weakform} restricted to functions $V$ in $E_k$, namely we seek $U_k\in \mathcal{C}^2([0,T],E_k)$ that satisfies

\begin{empheq}{align}\label{weakformwk}
\left\{
\begin{array}{r}
({U_k'',V})_0 + \delta ({U_k',V})_0 + ({U_k,V})_2 +\phi(U_k) ({(U_k)_x, V_x})_0 =({F ,V})_0\\
U_k(0) = Q_k U_0, \quad U_k'(0)= Q_k V_0\
\end{array}\right.
\end{empheq}
for all $V\in E_k$. We can write the coordinates of $U_k$ in the basis $(w_i)$, given by $u^k_i = (w_i, U_k)_0$, as functions and derive
from \eqref{weakformwk}, thanks to the relation between the usual and the buckling eigenvalue problems, the fairly simple systems of ODE's, for $i=1, \ldots, k$,

\begin{empheq}{align}\label{pvigik}
\left\{
\begin{array}{r}
{u^k_i}''(t) + \delta{u^k_i}'(t)+ \lambda_i u^k_i(t) +\Phi_i^k(u_1^k, \dots, u_k^k) = ({F(t), w_i})_0\,,\\
u^k_i(0)= (w_i, U_0)_0, \quad {u^k_i}'(0)=(w_i, V_0)_0\,,
\end{array}\right.
\end{empheq}
where the coupling terms $\Phi_i^k:\mathbb{R}^k\to\mathbb{R}$ are defined by
\begin{empheq}{align}     \label{definingPhik}
y^k = (y^k_1,\cdots, y^k_k), \quad
\Phi_i^k(y^k) =\Phi^k(y^k) \frac{\lambda_i}{\Lambda_i}y_i^k, \quad   \Phi^k(y^k) =  -P+S\sum_{j=1}^{k}\frac{\lambda_j}{\Lambda_j}(y_j^k)^2\,.
\end{empheq}

For each $i = 1, \ldots, k$, since $\Phi_i^k$ is smooth on $\mathbb{R}^k$, from the classical theory of ODEs, we know that \eqref{pvigik} has a unique solution, which can be extended to its maximal interval of existence $[0,T_{k})$ with $T_k \leq T$. Therefore \eqref{weakformwk} has a unique solution $U_k$, given by
\begin{empheq}{align}
U_k = \sum_{i=0}^ku^k_i w_i \, .
\end{empheq}

\vspace{10pt}\noindent
{\em Step 2. }{Some uniform estimates on $(U_k)$.}

The solution $U_k$, found in \textit{Step 1}, is $\mathcal{C}^2([0,T_k),E_k)$. Therefore we can take, for
each $t\in[0,T_k)$, the function $U_k'(t)\in E_k$ as test function in \eqref{weakformwk} to get
\begin{empheq}{align}\label{testuprime}
(U_k'', U_k')_0 + \delta (U_k', U_k')_0 + (U_k,U_k')_2 +\phi(U_k) \left((U_k)_x, (U_k')_x\right)_0 =({F ,U_k'})_0 \, .
\end{empheq}
Integrating \eqref{testuprime} over $[0,t]$ we find that
\begin{empheq}{align}\label{energyid-wk}
\frac{\|U_k'(t)\|^2_0}2+\frac{\|U_k(t)\|^2_2}2
-\frac{P\|(U_k)_x(t)\|^2_0}2+ \frac{S\|(U_k)_x(t)\|^4_0}4  +\delta\int_0^t \|U_k'\|_0^2 \hspace{.5cm}\\
=\int_0^t ({F ,U_k'})_0+\frac{\|U_k'(0)\|^2_0}2 +\frac{\|U_k(0)\|^2_2}2 -\frac{P\|(U_k)_ x(0)\|^2_0}2 +\frac{S\|(U_k)_x(0)\|^4_0}4 \,.
\end{empheq}
From \eqref{energyid-wk} and H\"older's inequality, we infer that
\begin{empheq}{multline}  \label{ineq1-quase}
\frac12\|{U_k'(t)}\|_0^2+\frac12\|{U_k(t)}\|_2^2 -\frac{P}2\|{(U_k)_ x(t)}\|_0^2+ \frac{S}4\|{(U_k)_x(t)}\|_0^4 \\
\leq \frac12\| {U_k'(0)}\|_0^2  +\frac12\|{ U_k(0)}\|_2^2 -\frac{P}2{\| {(U_k)_x(0)}\|_0^2} + \frac{S}4\| {(U_k)_x(0)}\|_0^4+\frac1{2}\int_0^t \|{F}\|_0^2 +\left(\frac12+|\delta|\right)\int_0^t \| {U_k'}\|_0^2\\
\leq\frac12\|{ U_0}\|_2^2+ \frac12\|{ V_0}\|_0^2+ \frac{C}4\|{ U_0}\|_2^4+\frac1{2}\int_0^t \| {F}\|_0^2+\left(\frac12+|\delta|\right)\int_0^t\| {U_k'}\|_0^2\, ,
\end{empheq}
where the second inequality is a consequence of the embedding $H^2_*(\Omega) \hookrightarrow H^1_*(\Omega)$.
Then, using that the maximum of $\tau \longmapsto \frac{P}{2} \tau^2 - \frac{S}{4} \tau^4$ is $\frac{P^2}{4S}$, from \eqref{ineq1-quase} we get
$$
\frac12\|{U_k'(t)}\|_0^2+\frac12\|{U_k(t)}\|_2^2 \leq \frac {P^2} {4S} +\frac12\|{ U_0}\|_2^2+
\frac12\|{ V_0}\|_0^2+ \frac{C}4\|{ U_0}\|_2^4+\frac1{2}\int_0^t \| {f}\|_0^2 +\left(\frac12+|\delta|\right)\int_0^t \| {U_k'}\|_0^2  \, .
$$
Hence, by the Gronwall inequality, we infer that
\begin{equation}\label{id-important:Arzela-Ascoli}
\frac12\|{U_k'(t)}\|_0^2+\frac12\|{U_k(t)}\|_2^2\leq \exp((1+2|\delta| )T)\left( \frac {P^2} {4S} +
\frac12\|{U_0}\|_2^2+ \frac12\|{ V_0}\|_0^2+ \frac{C}4\|{ U_0}\|_2^4+ \frac{1}{2}\int_0^t \|{ f}\|_0^2\right) \, .
\end{equation}
Since the right hand side is uniformly bounded  with respect to $k$ and $t \in [0,T_k)$,  we conclude that the solutions $(U_k)$ and their
derivatives $(U_k')$ are bounded with respect to $t$, and therefore the local solutions $U_k$ are actually defined on $[0,T]$.
We also infer from \eqref{id-important:Arzela-Ascoli} that $(U_k)$ is bounded in $\mathcal{C}^0([0,T],H^{2}_*(\Omega)) \cap \mathcal{C}^1([0, T ], L^2(\Omega))$. Moreover,  if follows from \eqref{id-important:Arzela-Ascoli} and from the compact embedding $H^2_*(\Omega) \hookrightarrow L^2(\Omega)$ that $(U_k)$ is equicontinuous from $[0,T]$ to $L^2(\Omega)$ and that $(U_k(t))$ is pre-compact in $L^2(\Omega)$ for every $t \in [0,T]$. Then, by the Ascoli-Arzel\'a Theorem, there exists a convergent subsequence $U_k \to U$ in $\mathcal{C}^0([0,T],L^2(\Omega)) $.

\vspace{10pt}\noindent
\textit{Step 3. }$(U_k)$ is a Cauchy sequence  in $\mathcal{C}^0([0,T],H^{2}_*(\Omega)) \cap\mathcal{C}^1([0,T],L^2(\Omega))$.

For integers $n>m>0$ define $U_{m,n} = U_n-U_m$ and $Q_{m,n} = Q_n-Q_m$. Testing \eqref{weakformwk} for $k=m$ with $v=(Q_m U_n - U_m)'$ and for
$k=n$ with $U_{m,n}'$ and subtracting these equations, we get
\begin{multline}\label{test-umnprime}
({U_{m,n})_0'',U_{m,n}'}+\delta({U_{m,n}',U_{m,n}'})_0+({U_{m,n},U_{m,n}'})_2-P({(U_{m,n})_x,(U_{m,n})_x'})_0 \\+S \|{(U_n)_x}\|_0^2({(U_n)_x, (U_{m,n})_x' })_0 -S \|{(U_m)_x}\|_0^2({(U_m)_x, (U_{m,n})_x' })_0= (Q_{m,n}F, U_n')_0 \,.
\end{multline}
Then we integrate the latter over $[0,t]$ to get
\begin{empheq}{align}\label{cauchyy}
\frac12\|{U'_{m,n}}\|_0^2+\frac12\|{U_{m,n}}\|_2^2+\delta \int_0^t\|{U'_{m,n}}\|_0^2=\frac P2\|{(U_{m,n})_x}\|_0^2 + {\mathcal E}(Q_{m,n}U_0,Q_{m,n}V_0) - \frac{S}{4}\|{(U_{m,n})_x}\|_0^4\\
 - \frac{S}{2} \int_0^t\|{(U_n)_x}\|_0^2 \left( \|{(U_{m,n})_x}\|_0^2 \right)'
- S \int_0^t ({(U_{m,n})_x, (U_n+U_m)_x})_0 ({(U_m)_x, (U_{m,n})_x'})_0 + \int_0^t(Q_{m,n}F, U_n')_0 \, .
\end{empheq}

From \eqref{id-important:Arzela-Ascoli}, we know that $(U_n')$ is bounded in $C([0,T], L^2(\Omega))$ and, since $F\in \mathcal{C}^0([0,T],L^2(\Omega))$, we conclude that
\begin{empheq}{align}\label{eq:estconv1}
\left|\int_0^t(Q_{m,n}f, U_n')_0 \right|  \leq C \int_0^T \|{Q_{m,n}f}\|  = o(1)\quad \forall\, t \in [0,T] \, .
\end{empheq}
On the other hand, we also have that
\begin{empheq}{align}\label{eq:estconv2}
{\mathcal E}(Q_{m,n}U_0,Q_{m,n}V_0) = o(1) \quad \text{uniformly w.r.t.} \ \ t \in [0,T] \, .
\end{empheq}

Now we recall the following interpolation inequality: there exists a positive constant $C$ such that
\begin{empheq}{align}\label{eq:interpolationineq}
\|{U_x}^2\| \leq C \|{U}\|_0 \|{U}\|_2 \quad \forall\, U \in H^{2}_*(\Omega) \, .
\end{empheq}
Then, we infer from \eqref{eq:interpolationineq} and \textit{Step 2} that
\begin{empheq}{align}\label{eq:estconv3}
\|{(U_{m,n})_x}\|_0^2 = o(1) \ \ \text{and} \ \ \|{(U_{m,n})_x}\|_0^4 = o(1) \ \ \text{uniformly w.r.t.} \ \ t \in[0,T] \, .
\end{empheq}
Next, from \eqref{eq:interpolationineq} and \textit{Step 2}, we infer that
\begin{eqnarray*}
\left| ({(U_{m,n})_x, (U_n+U_m)_x})_0 \right| &\leq& \|{(U_{m,n})_x}\|_0 \|{(U_n+U_m)_x}\|_0
\leq C \|{U_{m,n}}\|_0\|{U_{m,n}}\|_2 \|{U_n+U_m}\|_0\|{U_n+U_m}\|_2\\
& =& o(1) \ \ \text{uniformly w.r.t.} \ \ t \in [0,T] \, .
\end{eqnarray*}
Moreover,
$$
\int_0^t \left| ({(U_m)_x,(U_{m,n})_x'})_0\right| \leq \int_0^T \left| \int_{-l}^l \int_0^\pi (U_m)_x (U'_{m,n})_x \right|
 =  \int_0^T \left| \int_{-l}^l \int_0^\pi ({U_m})_{xx} U'_{m,n} \right| \leq\int_0^T \|{U_m}\|_2 \|{U'_m-U'_n}\|_0 \, ,
$$
which turns out to be uniformly bounded in $m,n$ from \eqref{id-important:Arzela-Ascoli}. Hence, by combining the last two inequalities, we obtain that
\begin{empheq}{align}\label{eq:estconv6}
 \int_0^t \left[({(U_{m,n})_x, (U_n+U_m)_x})_0 ({(U_m)_x, (U_{m,n})_x'})_0 \right]= o(1) \ \ \text{uniformly w.r.t.} \ \ t \in [0,T] \, .
\end{empheq}

Now observe that \eqref{eq:interpolationineq} and \textit{Step 2} guarantee that $(\|{(U_n)_x}\|_0)$ is uniformly bounded with respect to $t\in[0,T]$ and $n \in \mathbb{N}$. Hence, from \eqref{id-important:Arzela-Ascoli}, \eqref{eq:interpolationineq}, \eqref{eq:estconv3} we infer that
\begin{eqnarray*}
& &\left|\int_0^t\|{(U_n)_x}\|_0^2\left(\|{(U_{m,n})_x}\|_0^2\right)'\right|=2\left|\int_0^t\|{(U_n)_x}\|_0^2 ({(U_{m,n})_x, (U_{m,n})_x'})_0 \right| \\
& &= 2 \left| \int_0^t\|{(U_n)_x}\|_0^2 ({(U_{m,n})_{xx}, (U_{m,n})'})_0 \right| \leq C \int_0^t \|{(U_{m,n})_{xx}}\|_0 \|{U_{m,n}'}\|_0
\leq C \int_0^t \left[\|{U_{m,n}}\|_2^2 + \|{U_{m,n}'}\|_0^2\right] \, .
\end{eqnarray*}
Combined with \eqref{cauchyy}, \eqref{eq:estconv1}, \eqref{eq:estconv2}, \eqref{eq:estconv3}, and \eqref{eq:estconv6}, this enables us to we conclude that
$$
\|{U'_{m,n}}\|_0^2+ \|{U_{m,n}}\|_2^2 \leq C   \int_0^t \left(\|{U_{m,n}'}\|_0^2+\|{U_{m,n}}\|_2^2 \right)+ o(1)\quad\text{uniformly w.r.t.}\quad t\in[0,T] \,.
$$
By combining the latter with the Gronwall inequality, we infer that $(U_k)$ is a Cauchy sequence in $\mathcal{C}^0([0,T],H^{2}_*(\Omega))$ and in $\mathcal{C}^1([0,T],L^2(\Omega))$.
Moreover, at \textit{Step 2}, we proved that, up to a subsequence, $U_k \to U$ in $\mathcal{C}^0([0,T],L^2(\Omega))$. Then we conclude that, $U \in \mathcal{C}^0([0,T],H^{2}_*(\Omega)) \cap \mathcal{C}^1([0,T],L^2(\Omega))$, up to a subsequence,
\begin{equation}\label{ukconvergence}
U_k \rightarrow U \ \ \text{in} \ \ \mathcal{C}^0([0,T],H^{2}_*(\Omega)) \cap \mathcal{C}^1([0,T],L^2(\Omega)) \quad \text{as} \quad k \rightarrow \infty \, .
\end{equation}

\vspace{5pt}
\noindent
{\em Step 4. }The limit function $U$ is a solution of  \eqref{enlm} on the interval $[0,T]$.

Now take $v\in H^{2}_*(\Omega)$ and consider the sequence of projections $(Q_kv)$. Taking
$Q_kv$ as test function in \eqref{weakformwk} we get
$$
({U_k'', Q_kv})_0 + \delta ({U_k', Q_kv})_0 + ({U_k,Q_kv})_2 +\phi(U_k) ({(U_k)_x, Q_kv_x})_0 -({F ,Q_kv})_0=0 \, .
$$
Multiplying this identity by $\varphi\in\mathcal{C}^\infty_c(0,T)$ and integrating over $[0,T]$ we get
$$\int_0^T \Big( ({U_k'', Q_kv})_0 + \delta ({U_k', Q_kv})_0 + ({U_k,Q_kv})_2
+\phi(U_k) ({(U_k)_x, Q_k v_x})_0 -({F ,Q_kv})_0\Big)\varphi =0 \, .
$$
Integration by parts of the first term gives
$$
\int_0^T({U_k', Q_kv})_0\varphi' =\int_0^T\Big(  \delta ({U_k', Q_kv})_0 + ({U_k,Q_kv})_2
+\phi(U_k) ({(U_k)_x,Q_k v_x})_0 -({F ,Q_kv})_0\Big)\varphi 
$$
and, by letting $k\to\infty$, we get
\begin{empheq}{align}\label{uduaslinhas}
\int_0^T ({U',V})_0\varphi'=\int_0^T \Big( ({U,V})_2 +  \delta ({U',V})_0 + \phi(U) ({U_x,V_x})_0 -({F ,V})_0
\Big)\varphi  \, .
\end{empheq}
This shows that $U''\in\mathcal{C}^0([0,T],\mathcal{H})$ and $U$ solves the equation $U'' = -L U - \delta U' -\phi(U)U_{xx} + f$, where
$L : H^{2}_*(\Omega) \rightarrow \mathcal{H}$ stands for the canonical Riesz isometric isomorphism $\langle{Lw,z}\rangle := ({w,z})_2$ for all $w,z \in H^{2}_*(\Omega)$. Therefore
$U$ is indeed a solution of \eqref{weakform}.

\vspace{10pt}\noindent
{\em Step 5. }{Uniqueness for the linear problem, i.e. with $S=0$. }

Consider the linear problem obtained by taking $S=0$ in \eqref{enlm}. In this case, to prove uniqueness of the solution it suffices to prove that the trivial solution is the unique solution of
\begin{empheq}{align}\label{linearweak}
\left\{
\begin{array}{l}
{\langle U'',V \rangle} + \delta (U',V)_0 + (U,V)_2 -P{(U_x,V_x)_0} = 0,  \quad \forall \ V\in H^{2}_*(\Omega) \, , \vspace{5pt}\\
u(0) = 0, \ \ U'(0) = 0 \, .
\end{array}
\right.
\end{empheq}
Given $0\leq s\leq T$, define $V_s(t) = \int_t^s U(\tau) d\tau$ for $0\leq t \leq s$ and $V_s(t) = 0$ otherwise. Then $V_s(t)\in H^{2}_*(\Omega)$ and we can take it as a test function in the first equation of \eqref{linearweak} and integrate over $[0,T]$ to get
\begin{empheq}{align}\label{testvs}
\int_0^s \left[\langle{U''(t),V_s(t)}\rangle+\delta(U'(t), V_s(t) )_0+(U(t),V_s(t))_2-P ({(U(t))}_x{(V_s(t))_x})_0\right] dt=0 \, .
\end{empheq}

Now integrating by parts each term and taking into account that $U(0) = V_s(s)=0$ and $V_s'=-U$ on $[0,s]$, we rewrite

\begin{empheq}{align}
\int_0^s\langle{U'',V_s} \rangle =\left. \langle{U', V_s}\rangle \right |_0^s - \int_0^s\langle{U', V_s '}\rangle&\textcolor{black}{= \int_0^s(U,U')_0,}\\
\delta \int_0^s(U',V_s)_0 = \delta (U,V_s)|_0^s-\delta\int_0^s(U,V_s')_0 &= \delta\int_0^s (U,U)_0,\\
\int_0^s\big[(U,V_s)_2-P({U}_x,{(V_s)}_x)_0\big]&= \int_0^s\big[P ({(V_s)}_x, {(V_s')}_x)_0 -(V_s,V_s')_2\big]\, .
\end{empheq}

From \eqref{testvs} we infer
\begin{empheq}{align}\label{abc}
\|U(s)\|^2_0+\|V_s(0)\|^2_2-P\|(V_s(0))_x\|^2_0= -2\delta\int_0^s\|U\|^2_0\leq 0 \, .
\end{empheq}

Set $w(t) = \int_0^t U(\tau )d\tau = V_t(0)$ and we can estimate the $H^{1}_*(\Omega)$-norm of $w(s)$ using interpolation \cite{adams}. For every $\varepsilon>0$,
\begin{empheq}{align}\label{interpolation}
\| (w(s))_x\|_ 0^2 \leq \|w(s)\|_1^2 \leq C\|w(s)\|_0\|w(s)\|_2\leq \frac{C\varepsilon}2\|w(s)\|_2^2+ \frac{C}{2\varepsilon}\|w(s)\|_0^2\, .\end{empheq}

Now using \eqref{interpolation} in \eqref{abc} and taking $\varepsilon$ small enough, we can write
\begin{empheq}{align}
\|U(s)\|^2_0+(1- C'\varepsilon)\|w(s)\|^2_2\leq \frac{ C'}{\varepsilon} \|w(s)\|_0^2\leq C''\int_0^s \|U\|_0^2\, .\end{empheq}
Finally, the Gronwall inequality implies $U\equiv 0$.

\vspace{10pt}\noindent
{\em Step 6.}{ The energy identity \eqref{energyid} }

Consider first $S =0$. Then, from \eqref{weakform} and the uniqueness in the previous step, we obtain that $U$ is the limit of the sequence $(U_k)$
built in \textit{Step 1}. Thanks to \eqref{ukconvergence}, we can then take limit in \eqref{energyid-wk} to conclude that
\begin{empheq}{align}
\label{energyid-linear}
\frac{\|U'(t)\|^2_0}2 +\delta\int_0^t \|U'\|_0^2+\frac{\|U(t)\|^2_2}2
-\frac{P\|(U(t))_x\|^2_0}2\\
=\int_0^t (F,U)_0+\frac{\|U'(0)\|^2_0}2 +&\frac{\|U(0)\|^2_2}2 -\frac{P\|(U(0))_x\|^2_0}2 \, .
\end{empheq}
This establishes the energy identity \eqref{energyid} in this case.

Now consider $S>0$ and let $U$ be a weak solution of \eqref{enlm}.
Then for every $V\in H^{2}_*(\Omega)$, we can integrate by parts
\begin{empheq}{align}\label{partsLemma}
{(U_x,V_x)_0}= \int_{-l}^l \int_0^\pi U_x V_x dx\ dy = \int_{-l}^l\left( [U_x V]_0^\pi -\int_0^\pi U_{xx} V dx \right)dy=
-({U_{xx},V})_0 \, .
\end{empheq}
Using \eqref{partsLemma} we see that $U$ satisfies
$$
{\langle U'',V \rangle} + \delta (U',V)_0 + (U,V)_2 -P{(U_x,V_x)_0} = (g,V)_0,  \quad \forall \ V\in H^{2}_*(\Omega), \, \forall t \in (0,T) \, ,
$$
where $g =F + S\|{U_x}\|_0^2 U_{xx}\in \mathcal{C}^0([0,T], L^2(\Omega))$. We then conclude as in \eqref{energyid-linear} that
\begin{empheq}{align}
\frac{\|U'(t)\|^2_0}2 +\delta\int_0^t \|U'\|_0^2+\frac{\|U(t)\|^2_2}2
-\frac{P\|U(t)\|^2_1}2\\
=\int_0^t ({F ,U'})_0+\int_0^t({  S\|{U_x}\|_0^2 U_{xx},U'})_0+\frac{\|U'(0)\|^2_0}2 +&\frac{\|U(0)\|^2_2}2 -\frac{P\|U(0)\|^2_1}2 \, .
\end{empheq}
It remains to verify that
\begin{empheq}{align}\label{lastcheck}
S\int_0^t({  \|{U_x}\|_0^2 U_{xx},U'})_0
 = -\frac S4\|{U_x(t)}\|_0^4 +\frac S4\|{(U_0)_x}\|_0^4\, .\end{empheq}

To that end, consider the sequence $(Q_kU)$. Then take into account that $Q_kU \in\mathcal{C}^2([0,t],E_k)$ for every $k$ and that $Q_kU(t)\to U(t)$ in $H^{2}_*(\Omega)$ for every $t\in[0,T]$.
Integrating by parts as in \eqref{partsLemma} we infer that
$$
S\int_0^t({  \|{(Q_kU)_x}\|_0^2 (Q_kU)_{xx},(Q_kU)'})_0 = -\frac S4\|{(Q_kU)_x(t)}\|_0^4 +\frac S4\|{(Q_kU(0))_x}\|_0^4 \, . \hspace{15pt}
$$
Observe that the embedding $H^{2}_*(\Omega)\to H^{1}_*(\Omega)$ ensures that \begin{empheq}{align}
-\frac S4\|{(Q_kU)_x(t)}\|_0^4 +\frac S4\|{(Q_kU(0))_x}\|_0^4\to-\frac S4\|{U_x(t)}\|_0^4 +\frac S4\|{(U_0)_x}\|_1^4\, .\end{empheq}
Then, the Lebesgue Theorem yields the result. Indeed, for every $t\geq0$ we can estimate
\begin{empheq}{align}
\left| \|{(Q_kU)_x}\|_0^2({(Q_kU)_{xx},(Q_kU)'})_0 \right| & \leq \|{Q_kU}\|_1^2\|{Q_kU_{xx}}\|_0\|{Q_kU'}\|_0\leq C\|{Q_kU}\|_2^3\|{(Q_kU)'}\|_0\\
 & \leq C(1+\|{U(t)}\|_2^3\|{U'(t)}\|_0), \qquad\forall \, k,\, t \, ,\end{empheq}
by the Parseval's identity and  $C$ is a positive constant, that does not depend on $k$ or $t$. Now since from hypothesis $U\in\mathcal{C}^1([0,t],L^2(\Omega))\cap\mathcal{C}^0([0,t], H^{2}_*(\Omega))$, the function $H(s) =\|{U(s)}\|_2^3\|{U'(s)}\|_0$ is in $L^1([0,T])$, and we conclude that
\begin{empheq}{align}
S\int_0^t\|{Q_kU_x}\|_0^2({Q_kU_{xx},Q_kU'})_0 \to S\int_0^t({  \|{U_x}\|_0^2 U_{xx},U'})_0  \, .
\end{empheq}

\vspace{10pt}\noindent
{\em Step 7. } Uniqueness for the case of $S>0$.

Let $U$ and $w$ be weak solutions of \eqref{enlm}, that is, $U$ and $w$ satisfy \eqref{weakform}, $U(0) = w(0)=U_0$ and $U'(0)=w'(0)=V_0$. Set $z = U-w$, so that $z(0)=z'(0)=0$ and
$$
{\langle z'',V \rangle} + \delta ({z',V})_0 + ({z,V})_2 +\phi(U){(z_x,V_x)_0}
+(\phi(U) - \phi(w) ){(w_x,V_x)_0} =0, \quad \forall\, V\in H^{2}_*(\Omega) \, .
$$
We must prove that $z=0$. Let us rewrite the nonlinear term to apply our energy identity. Taking into account that
\begin{empheq}{align}
\phi(z){(z_x,V_x)_0} &= -P{(z_x,V_x)_0}+S\|{z_x}\|_0^2{(z_x,V_x)_0} = -P{(z_x,V_x)_0}+S(\|{U_x}\|_0^2+\|{w_x}\|_0^2-2{(U_x,w_x)_0}) {(z_x,V_x)_0}\\
&=\phi(U)  {(z_x,V_x)_0}+S(\|{w_x}\|_0^2-2{(U_x,w_x)_0}){(z_x,V_x)_0} \, ,
\end{empheq}
we see that $z$ satisfies the equation
\begin{empheq}{align}
{\langle z'',V \rangle} + \delta ({z',V})_0 + ({z,V})_2 +\phi(z){(z_x,V_x)_0}=({h,V})_0, \quad \forall \, V\in H^{2}_*(\Omega) \, ,
\end{empheq}
where, using integration by parts, $h\in\mathcal{C}^0([0,T],L^2(\Omega))$ is defined by
\begin{empheq}{align}
h= - S(\|{w_x}\|_0^2-2{(U_x,w_x)_0})z_{xx} + (\phi(U) - \phi(w) )w_{xx} \, .
\end{empheq}
From the energy identity \eqref{energyid} proved in the step before and using the initial data, we have for each $t\in[0,T]$,
\begin{empheq}{align}\label{uniq-step1new}
\frac12\|{z'}\|_0^2+\frac12\|{z}\|_2^2 +\frac14S\|{z_x}\|_0^4
= \frac12P\|{z_x}\|_0^2 + \int_0^t ({h,z'})_0
-\delta\int_0^t \|{z'}\|_0^2 \, .
\end{empheq}
We estimate the first norm on the right using interpolation \cite{adams}:
\begin{empheq}{align}\label{uniq-step2new}
 \|{z_x}\|_0^2 \leq \|{z}\|_1^2\leq C \|{z}\|_2 \|{z}\|_0 \leq C \varepsilon \|{z}\|_2^2+\frac C\varepsilon\|{z}\|_0^2 =
C\varepsilon\|{z}\|_2^2 +\frac C{2\varepsilon}\int_0^t({z,z'})_0 \\
\leq C\varepsilon\|{z}\|_2^2 +\frac C{4\varepsilon}\int_0^t( \|{z}\|_0^2+ \|{z'}\|_0^2)
\leq C\varepsilon\|{z}\|_2^2 +\frac{ \bar{ C}}{4\varepsilon}\int_0^t(\|{z}\|_2^2+&\|{z' }\|_0^2 )\, .
\end{empheq}

We also have
\begin{empheq}{align}\label{uniq-step3quase}
\|{h}\|_0^2 &\leq 2S^2 \, |\|{w_x}\|_0^2-2{(U_x,w_x)_0}|^2\, \|{z_{xx}}\|_0^2+2 |\phi(U) - \phi(w) |^2\,\|{w_{xx} }\|_0^2\\
&\leq C( \|{z}\|_2^2+(\|{U_x}\|_0^2-\|{w_x}\|_0^2)^2 )
\leq C(\|{z}\|_2^2+\|{U_x+w_x}\|_0^2\|{z_x}\|_0^2)  \leq C\|{z}\|_2^2\,.
\end{empheq}
This inequality yields
\begin{empheq}{align}
\label{uniq-step3new}
\left|\int_0^t({h,z'})_0\right| \leq\frac12\int_0^t\left(\|{h}\|_0^2+\|{z'}\|_0^2\right)
\leq C\int_0^t \|{z}\|_2^2 + \frac12 \int_0^t \|{z'}\|_0^2 \, .
\end{empheq}

Combining \eqref{uniq-step1new}, \eqref{uniq-step2new} and \eqref{uniq-step3new} and choosing an appropriate value for $\varepsilon$,
we conclude
\begin{empheq}{align}\label{ineqDomingo}
\|{z'}\|_0^2+\|{z}\|_2^2
\leq C \int_0^t \left(\|{z'}\|_0^2 + \|{z}\|_2^2\right) \,.
\end{empheq}
Then, the Gronwall inequality guarantees that $z\equiv 0$.

\section{Proof of the asymptotic behavior under stationary loads} \label{section:asymptotic}

Throughout this section we restrict our study to stationary loads, more precisely, we assume that $F\in L^2(\Omega)$ is time-independent. In this case, we can write \eqref{physicsenergyid} as
\begin{empheq}{align}
\label{energyid-fixedf}
{\mathcal E}_{F}(U;t)+\delta\int_0^t \|{U'}\|_0^2={\mathcal E}_{F}(U;0) \,,
\end{empheq}
with
\begin{empheq}{align}\label{eq:someenergies}
{\mathcal P}_{F}(U;t) =  {\mathcal P}(U;t) - ({F ,U(t)})_0\,, \quad {\mathcal E}_{F}(U;t)={\mathcal K}(U;t)+{\mathcal P}_{F}(U;t) = {\mathcal E}(U;t)- ({F ,U(t)})_0\,,
\end{empheq}
where ${\mathcal P}(U;t)$, ${\mathcal K}(U;t)$ and ${\mathcal E}(U;t)$ correspond respectively to the potential, kinetic and mechanical energies as defined in Section \ref{section:weel-posedeness}, and we readily see that the energy ${\mathcal E}_{F}(U;t)$ of the solution is non-increasing. We also introduce the functional
${\mathcal P}_F: H^{2}_*(\Omega) \to \mathbb{R}$ defined as
\[
{\mathcal P}_F(W)=\frac12\|{W}\|_2^2-\frac P2\|{W_x}\|_0^2+\frac S4\|{W_x}\|_0^4 - (F,W)_0\,, \quad W \in H^{2}_*(\Omega).
\]

\begin{lemma} \label{uprimetozero}
Let $\delta>0$, $P,S > 0$, $U_0\in H^{2}_*(\Omega)$, $V_0\in L^2(\Omega)$ and $F \in L^2(\Omega)$, and denote by $U$  the solution of \eqref{enlm}. Then:
\begin{enumerate}[i)]
\item
$(U,U') \in L^\infty([0,\infty),H^{2}_*(\Omega)) \times L^{\infty}([0, \infty), L^2(\Omega)) $.
\item $U'\in L^2(\Omega \times(0, \infty)) $.
\end{enumerate}
\end{lemma}

\begin{proof}
Recall from \eqref{energyid} that
\begin{empheq}{align}\label{energythiscase}
 \frac12\|{U'}\|_0^2 +  &\frac12\|{U}\|_2^2-\frac P2\|{U_x}\|_0^2+\frac S4\|{U_x}\|_0^4  + \delta\int_0^t \|{U'}\|_0^2 = {\mathcal E}(U_0,V_0) +
\int_0^t ({F ,U'})_0\,.\hspace{10pt}
\end{empheq}
Since $F$ is time-independent, we can estimate
\begin{empheq}{align}\label{estimar}
\left| \int_0^t ({F ,U'})_0 \right| = \left|({F ,U(t)})_0 - ({F ,U_0})_0 \right| \leq \frac{\lambda_1}{4}\|{U}\|_0^2 + \frac1{\lambda_1}\|{F}\|_0^2 + \left| ({F ,U_0})_0\right|\,.
\end{empheq}
From \eqref{energythiscase}, \eqref{estimar} and the characterization of $\lambda_1$ we conclude that
\begin{empheq}{align}
\frac12\|{U'}\|_0^2 +\frac14\|{U}\|_2^2-\frac P2\|{U_x}\|_0^2+\frac S4\|{U_x}\|_0^4  + \delta\int_0^t \|{U'}\|_0^2 \leq
 \frac1{\lambda_1}\|{F}\|_0^2 +\left| ({F ,U_0})_0\right|+&{\mathcal E}(U_0,V_0)\, .
\end{empheq}
Then, using that the maximum of $\tau \longmapsto \frac{P}{2} \tau^2 - \frac{S}{4} \tau^4$ is $\frac{P^2}{4S}$, we infer that
\begin{empheq}{align}\label{enfim}
\frac12\|{U'}\|_0^2 + \frac14 \|{U}\|_2^2 + \delta\int_0^t \|{U'}\|_0^2 \leq
 \frac1{\lambda_1}\|{F}\|_0^2& +\left|({F ,U_0})_0\right|+{\mathcal E}(U_0,V_0) +\frac{P^2}{4S} \, .
\end{empheq}
Observing that the right hand side in \eqref{enfim} is independent of $t$, we conclude the proof.
\end{proof}

In the next lemma we establish the convergence of the solution.

\begin{lemma}\label{convergencelemma}
Let $\delta>0$, $S > 0$, $U_0\in H^{2}_*(\Omega)$, $V_0\in L^2(\Omega)$, $F \in L^2(\Omega)$,
and denote by $U$  the solution of \eqref{enlm}.
\begin{enumerate}[i)]
\item Then there exist $\overline U\in H^{2}_*(\Omega)$ a solution to \eqref{stationary} and an increasing sequence $(t_n)$ such that
\begin{empheq}{align}\label{resultadofraco}
U(t_n) \to \overline U\text{ in }H^{2}_*(\Omega), \quad U'(t_n)\to 0\text{ in }L^2(\Omega), \quad \quad n\to\infty \, .
\end{empheq}
\item
If $\overline U\in H^{2}_*(\Omega)$ also satisfies
\begin{empheq}{align}\label{condextra}
{\mathcal P}_{F}(\overline U) \leq \inf_{t \in[0,\infty)} {\mathcal P}_{F}(U;t) \, ,
\end{empheq}
then a stronger convergence holds:
\begin{empheq}{align}\label{resultadoforte}
U(t) \to \overline U\text{ in }H^{2}_*(\Omega), \quad U'(t)\to 0\text{ in }L^2(\Omega), \quad \quad t\to\infty \, .
\end{empheq}
\end{enumerate}
\end{lemma}
\begin{proof} {\em Proof of i) }
Let $(s_n)$ be a sequence of positive numbers such that $s_n\to\infty$ and
\begin{empheq}{align}\label{maiorque2} 3 \geq s_{n+1}-s_n\geq2, \ \ \forall \, n\, .
\end{empheq}
Then, by Lemma \ref{uprimetozero},
\begin{empheq}{align}\label{serieconverge}
\sum_n \int_{s_n}^{1+s_n} \|{U'}\|_0^2 \leq \int_0^\infty  \|{U'}\|_0^2 < \infty \, ,
\end{empheq}
which implies that
\begin{empheq}{align}\label{aaaa}
\int_{s_n}^{1+s_n} \|{U'}\|_0^2 \to 0, \ \ \text{as} \ \ n \to \infty\,.
\end{empheq}
So, for each $n \in {\mathbb{N}}$, there exists $t_n \in(s_n, 1+s_n)$ such that
\begin{empheq}{align}\label{cond1}\|{U'(t_n)}\|_0^2 = \int_{s_n}^{1+s_n} \|{U'}\|_0^2 \to 0\,.\end{empheq}
From \eqref{maiorque2}, $(t_n)$ is increasing and $1\leq t_{n+1} - t_n\leq 4$. This allows us to argue as in \eqref{serieconverge} and conclude that
\begin{empheq}{align}\label{cond2}
\int_{t_n}^{t_{n+1}}  \|{U'}\|_0^2 \to 0 \, .
\end{empheq}
Then, for every $v\in H^{2}_*(\Omega)$, from \eqref{cond2},
\begin{empheq}{align}\label{inequprime}
\left| \int_{t_n}^{t_{n+1}}  \delta ({U',V})_0 \right| \leq\delta\int_{t_n}^{t_{n+1}} \|{U'}\|_0 \|{V}\|_0 \leq\delta\|{V}\|_0 \int_{t_n}^{t_{n+1}} \|{U'}\|_0
\leq 2 \,\delta\|v\|_0  \left( \int_{t_n}^{t_{n+1}} \|{U'}\|_0^2  \right)^{1/2} \to 0 \,.
\end{empheq}
Moreover, from \eqref{cond1},
\begin{empheq}{align}   \label{ineq53}
\left| \int_{t_n}^{t_{n+1}}{\langle U'',V \rangle} \right| = \left| ({U'(t_{n+1}),V})_0- ({U'(t_{n}),V})_0 \right|  \leq  \|{V}\|_0(\|{U'(t_{n})}\|_0+\|{U'(t_{n+1})}\|_0) \to 0 \,.
\end{empheq}
The last two inequalities yield
\begin{empheq}{align}\label{timeparttozero}
\int_{t_n}^{t_{{n}+1}} \left[ {\langle U'',V \rangle}+\delta ({U',V})_0\right] \to0, \quad n\to\infty \,.
\end{empheq}
So, given $V\in H^{2}_*(\Omega)$, for each $n$ there exists $t_n^v \in (t_n,t_{n+1})$ such that
\begin{empheq}{align}\label{timepartgoestozero}
\langle{U''(t_n^v),V} \rangle+ \delta ({U'(t_n^v),V})_0 = (t_{n+1} - t_n)^{-1}\int_{t_n}^{t_{n+1}} \left[{\langle U'',V \rangle} + \delta ({U',V})_0\right] \to 0
\, .\end{empheq}
Since $(U(t_n^v))$ is a bounded sequence in $H^{2}_*(\Omega)$,  there exists a subsequence $(n_k) \subset {\mathbb{N}}$, with $t_{n_k}^v \in (t_{n_k}, t_{n_k + 1})$, such that $U(t_{n_k}^v) \rightharpoonup \overline{U}_v$ in $H^{2}_*(\Omega)$. Let us prove that $\overline{U}_v$ does not depend on $V$. Let $V,W\in H^{2}_*(\Omega)$. Since  $(U(t_n^v))$ and  $(U(t_n^w))$ are bounded sequence in $H^{2}_*(\Omega)$, there is exists a common subsequence $(n_k)$ such that $U(t_{n_k}^v) \rightharpoonup \overline{U}_v$ and $U(t_{n_k}^w) \rightharpoonup \overline{U}_w$ in $H^{2}_*(\Omega)$. Then
\begin{empheq}{align}\label{coisa1}
\|{U(t_{n_k}^v)-U(t_{n_k}^w)}\|_0^2 =\int_\Omega \left(U(t_{n_k}^v) - U(t_{n_k}^w) \right)^2 = \int_\Omega
\left( \int_{t_{n_k}^v}^{t_{n_k}^w} U'\right)^2 \leq \\
\leq \int_\Omega\left|\int_{t_{n_k}^v}^{t_{n_k}^w} U'^2\right|\left|\int_{t_{n_k}^v}^{t_{n_k}^w} 1\right| \leq  4\int_\Omega \left|\int_{t_{n_k}^v}^{t_{n_k}^w} U'^2 \right|
\leq  4\int_{t_{n_k}}^{t_{n_k +1}} \|{U'}\|_0^2 \, ,
\end{empheq}
and from \eqref{aaaa}, taking the limit in $k$, we get $\overline{U}_v = \overline{U}_w$.  Therefore we drop the subscript $v$ and we denote this common limit as $\overline U$.

We now show that $\overline U$ is a solution of the stationary problem \eqref{stationary}.
Take $V\in H^{2}_*(\Omega)$. Since, up to a subsequence, $U(t_n^v) \rightharpoonup \overline{U}$ in $H^{2}_*(\Omega)$,
\begin{empheq}{align}({U(t_n^v),V})_2\to ({\overline{U},V})_2 \,.\end{empheq}
From the compact embedding $H^{2}_*(\Omega) \hookrightarrow H^{1}_*(\Omega)$,
\begin{empheq}{align}\label{strongh1convergence}
\|{(U(t_n^v))_x}\|_0^2 \to \|{(\overline U)_x}\|_0^2 \quad\textrm{and}\quad ({(U(t_n^v))_x,V_x})_0\to ({(\overline U)_x, V_x})_0\, .\end{empheq}
Now from \eqref{timeparttozero}, \eqref{timepartgoestozero}, and the above convergences, by taking limit in \eqref{weakform}, we infer
that $\overline U$ is a weak solution of \eqref{stationary}, namely,
\begin{empheq}{align}\label{weakstationary}
({\overline U,V})_2 + \phi(\overline U){(\overline{U}_x,V_x)_0} = ({F ,V})_0 \, .
\end{empheq}

\vspace{10pt}
Subtracting \eqref{weakstationary} from \eqref{weakform} and writing $w(t)=U(t)-\overline U$, we get
\begin{empheq}{align}\label{idw00}
\!\!\!\!\langle {w'',V} \rangle+ \delta({w',V})_0+({w,V})_2 -P{(w_x,V_x)_0}
+S\|{U_x}\|_0^2{(U_x,V_x)_0}- S\|{\overline U_x}\|_0^2 ({\overline U}_x, &{ V}_x)_0 =0, \ \forall \,V \in H^{2}_*(\Omega) \, .
\end{empheq}
Take $w$ as test function and integrate \eqref{idw00} over $(t_n,t_{n+1})$
\begin{empheq}{align} \label{idw0}
\int_{t_{n}}^{t_{n+1}} \left[\langle {w'',w}\rangle  + \delta({w',w})_0 + \|{w}\|_2^2-P \|{w_x}\|_0^2
+ S\|{U_x}\|_0^2({U_x,w_x})_0 -S\|{\overline U_x}\|_0^2({\overline U_x,w_x})_0\right]=0 \, .
\end{empheq}

Integrating by parts,
$$
\int_{t_{n}}^{t_{n+1}}\left[\langle{w'',w} \rangle+ \delta({w',w})_0 \right]= ({w'(t_{n+1}),w(t_{n+1})})_0 - ({w'(t_n),w(t_n)})_0
+\int_{t_{n}}^{t_{n+1}} \left[-\|{w'}\|_0^2 +\delta\langle{w',w} \rangle\right] \, ,
$$
and so
\begin{empheq}{align} \label{extremamentelonga}
&\left|\int_{t_{n}}^{t_{n+1}}\left[\langle{w'',w}\rangle + \delta({w',w})_0\right]\right| \leq  | ({w'(t_{n+1}),w(t_{n+1})})_0 |+ |({w'(t_n),w(t_n)})_0|+ \left|
\int_{t_{n}}^{t_{n+1}} \left[-\|{w'}\|_0^2 +\delta\langle{w',w}\rangle\right] \right|  \\
& \leq  \sup_{k\in{\mathbb{N}}} \|{w(t_k)}\|_0( \|{U'(t_n)}\|_0+\|{U'(t_{n+1})}\|_0  ) + \int_{t_{n}}^{t_{n+1}} \left[\|{U'}\|_0^2 +\delta |\langle{w',w}\rangle| \right]\\
& \leq   \sup_{k\in{\mathbb{N}}} \|{w(t_k)}\|_0( \|{U'(t_n)}\|_0+\|{U'(t_{n+1})}\|_0  )
 + \int_{t_{n}}^{t_{n+1}} \|{U'}\|_0^2 + 2 \, \delta\sup_{t\in[0,\infty)} \|{w}\|_0\left(\int_{t_{n}}^{t_{n+1}}  \|{U'}\|_0^2\right)^{1/2} \, .
\end{empheq}
By combining the latter with \eqref{cond1} and \eqref{cond2}, we infer from \eqref{idw0} that
\begin{empheq}{align}\label{cond3}
\int_{t_n}^{t_{n+1}} \left[\|w^2\|_2 -P \|{w_x}\|_0^2+S\|{U_x}\|_0^2{(U_x,w_x)_0} -S\|{\overline U_x}\|_0^2 ({\overline U}_x,{w}_x)_0\right] \to 0 \, .
\end{empheq}
Hence, adding \eqref{cond2} and \eqref{cond3}, there exists $\overline t_n\in (t_n,t_{n+1})$, such that
\begin{empheq}{align}   \label{indoprazero77}
\!\!\!\|{U'(\overline t_n)}\|_0^2+\!\| {w(\overline t_n)}\|_2^2 \!-\!P \|{w(\overline t_n)_x}\|_0^2+\!S\|{U(\overline t_n)_x}\|_0^2({U(\overline t_n)}_x,{w(\overline t_n)}_x)_0
\!-\!S\|{\overline U(\overline t_n)_x}\|_0^2({\overline U(\overline t_n)}_x,{w(\overline t_n)}_x)_0 \!\to \! 0 \,.
\end{empheq}

Let us prove that $U(\overline t_n)\to \overline U$ in $H^{1}_*(\Omega)$ as $n\to\infty$. Indeed, consider any subsequence $(\overline t_{n_k})$.  From the boundedness of $U$ in $H^{2}_*(\Omega)$, a subsubsequence of $(U(\overline t_{n_{k}}) )$, denoted $(U(\tau_j))$, converges weakly in $H^{2}_*(\Omega)$ and strongly in $H^{1}_*(\Omega)$ to some $\overline{\overline U}$. Arguing as in \eqref{coisa1}, we show that $\overline{U} = \overline{\overline U}$. Then $U(\tau_j)\to \overline U$ in $H^{1}_*(\Omega)$ as $j\to\infty$. Therefore, we infer that  $U(\overline t_n)\to \overline U$ in $H^{1}_*(\Omega)$ as $n\to\infty$.

Now from \eqref{indoprazero77} we obtain that
\begin{empheq}{align}
\|{U'(\overline t_{n{}})}\|_0^2+\| {w(\overline t_{n{}})}\|_2^2
=\Big(\|{U'(\overline t_{n{}})}\|_0^2+\|{w(\overline t_{n{}})}\|_2^2 -P \|{w(\overline t_{n{}})_x}\|_0^2\\
+S\|{U(\overline t_{n{}})_x}\|_0^2 ({U(\overline t_{n{}})}_x,{w(\overline t_{n{}})}_x)_0
-S\|{\overline U(\overline t_{n{}})_x}\|_0^2(&{\overline U(\overline t_{n{}})}_x,{w(\overline t_{n{}})}_x)_0 \Big) + \\
+ \Big(P \|{w(\overline t_{n{}})_x}\|_0^2-S\|{U(\overline t_{n{}})_x}\|_0^2({U(\overline t_{n{}})}_x,{w(\overline t_{n{}})}_x)_0
+S\|{\overline U(\overline t_{n{}})_x}\|_0^2({\overline U(\overline t_{n{}})}_x,&{w(\overline t_{n{}})}_x)_0 \Big)  \to 0 \, .
\end{empheq}
Therefore $U(\overline t_{n{}})\to \overline U$ in $H^{2}_*(\Omega)$ and $  U'(\overline t_{n{}})\to 0$ in $L^2(\Omega)$, which establishes i).

\medskip
Then, from \eqref{resultadofraco},
\begin{empheq}{align}\label{just-proved}
{\mathcal E}_{F}(U;\overline t_{n{}} ) \to  {\mathcal P}_{F}(\overline U) \, .
\end{empheq}
Since, according to \eqref{energyid-fixedf}, the energy $\mathcal{E}_{F}$ is monotonic, we infer from \eqref{just-proved}  that the convergence of the energy happens in the whole flow, that is,
\begin{empheq}{align}\label{just-proved-allt}
{\mathcal E}_{F}(U;t ) \to  {\mathcal P}_{F}(\overline U) \, .
\end{empheq}

\medskip
\noindent {\em Proof of ii) }
Next, assume that \eqref{condextra} holds. Then, from \eqref{just-proved-allt} and Lemma \ref{uprimetozero},
\begin{empheq}{align}\label{pot-converges}
{\mathcal P}_{F} (\overline U) = \limsup_{t\to\infty}\big( {\mathcal P}_{F}(U;t)+{\mathcal K}(U;t) \big)\geq {\mathcal P}_{F}(\overline U) +\limsup_{t\to\infty} {\mathcal K}(U;t) \, .
\end{empheq}
Therefore, $U'(t) \to 0$ in $L^2(\Omega)$ as $t\to\infty$. Moreover,
\begin{empheq}{align}\label{pfmenospf}
{\mathcal P}_{F}(U;t) - {\mathcal P}_{F}(\overline U) = {\mathcal E}_{F}(U;t) - {\mathcal K} (U;t) - {\mathcal P}_{F} (\overline U) \to 0, \quad t\to\infty \, .
\end{empheq}
Arguing as in \eqref{coisa1}, we infer that $U(t) \to \overline U$ in $L^2(\Omega)$ as $t \to \infty$.
Indeed, let $(t_n)$  be the increasing sequence from \eqref{resultadofraco}, such that $t_n\to\infty$, $U(t_n) \to \overline U$ in $H^{2}_*(\Omega)$. Recall that $0 < t_{n+1} - t_n \leq 8$ for all $n \in {\mathbb{N}}$. Given $t>0$ large, we must have $t\in[t_n,t_{n+1}]$ for some $n\in{\mathbb{N}}$.
 We then estimate
\begin{empheq}{align}\label{coisa2}
\|{U(t) -\overline U}\|_0^2 & \leq 2(\|{U(t)-U(t_{n})}\|_0^2+\|{U(t_n) - \overline U}\|_0^2)
=2\|{U(t_n) - \overline U}\|_0^2 +2\int_\Omega \left(U(t) - U(t_{n}) \right)^2 \\
& =2\|{U(t_n) - \overline U}\|_0^2  +2\int_\Omega( \int_{t_n}^{t} U')^2 \leq
2\|{U(t_n) - \overline U}\|_0^2+2\int_\Omega\left|\int_{t_n}^{t} U'^2\right|\left|\int_{t_n}^{t} 1\right|  \\
 & \leq 2\|{U(t_n) - \overline U}\|_0^2+16 \int_\Omega \left|\int_{t_{n}}^{t} U'^2 \right| 
\leq 2\|{U(t_n) - \overline U}\|_0^2+16 \int_{t_{n}}^{t_{n+1}} \|{U'}\|_0^2 \, .
\end{empheq}
Given $\varepsilon>0$, we can take $n_0>0$ such that
\begin{empheq}{align}
\|{U(t_n) - \overline U}\|_0^2 \leq \frac\varepsilon4, \quad \int_{t_{n}}^{t_{n+1}} \|{U'}\|_0^2\leq \frac\varepsilon{32}\ \ \forall \, n \geq n_0,
\end{empheq}
and the convergence follows from \eqref{coisa2}.

Since $U(t) \to \overline U$ in $L^2(\Omega)$ and $U$ is bounded in $H^{2}_*(\Omega)$, by standard interpolation, we infer that $U(t) \to \overline U$ in $H^{1}_*(\Omega)$ as $t\to\infty$. By \eqref{pfmenospf} and the strong convergence in $H^{1}_*(\Omega)$, we can write
$$
\frac12 \|{U}\|_2^2 -\frac12 \|{\overline U}\|_2^2 = {\mathcal P}_{F} (U) - {\mathcal P}_{F} (\overline U)
+\frac P2(\|{U_x}\|_0^2 -\|{\overline U_x}\|_0^2)
- \frac S4( \|{U_x}\|_0^4 - \|{\overline U_x}\|_0^4)\to 0, \quad t\to\infty\, ,
$$
that is,
\begin{empheq}{align}\label{normasconvergindo}
\|{U(t)}\|_2 \to \|{\overline U}\|_2, \quad \text{as} \quad t\to\infty\,.
\end{empheq}

Let $(r_n)$ be any sequence of positive number such that $r_n \to \infty$.  Given any subsequence $(r_{n_k}) = (s_k)$, since $(U(s_k))$ is bounded in $H^{2}_*(\Omega)$, and $U(s_k)\to \overline U$ in $H^{1}_*(\Omega)$, we have for a further subsequence, denoted by $s_{k_{j}} = \tau_j$,
\begin{empheq}{align}
U(\tau_j)\rightharpoonup \overline U, \quad \|{U(\tau_j)}\|_2 \to \|{\overline U}\|_2,
\end{empheq}
which in turn implies the strong convergence $U(\tau_j) \to \overline U$ in $H^{2}_*(\Omega)$. This shows that indeed $U(r_n) \to \overline U$ in $H^{2}_*(\Omega)$ as $n \to \infty$ and since the sequence $(r_n)$ was arbitrary, we have proved that
\begin{empheq}{align}
U(t) \to \overline U \ \ \text{in} \ \ H^{2}_*(\Omega) \ \ \text{as} \ \ t\to\infty\, .
\end{empheq}
\end{proof}

\begin{proof}[\textbf{Proof of Theorem \ref{thmcoercive}}]
Since $0\le P \le \Lambda_1$, the solution of $\eqref{stationary}$ is unique and realizes the minimum of ${\mathcal P}_{F}$ over $H^{2}_*(\Omega)$. Therefore \eqref{condextra} is satisfied and the proof follows from Lemma \ref{convergencelemma}.
\end{proof}

\begin{proof}[\textbf{Proof of Theorem \ref{oculos}}]
Let $\delta>0$, $S>0$,
$U_0\in H^{2}_*(\Omega)$, $V_0\in L^2(\Omega)$, $\Lambda_1<P\le\Lambda_2$ {}and $U$ be the solution of \eqref{enlm} with $F\equiv0$ and assume that ${\mathcal E}(U; 0) <0$. Lemma \ref{convergencelemma} guarantees the existence of $\overline U\in H^{2}_*(\Omega)$ a solution to \eqref{stationary} and an increasing sequence $(t_n)$, with $t_n \to \infty$, such that
\begin{empheq}{align}\label{resultadofraco++}
U(t_n) \to \overline U\text{ in }H^{2}_*(\Omega), \quad U'(t_n)\to 0\text{ in }L^2(\Omega), \quad \quad n\to\infty \, .
\end{empheq}
Then, by \eqref{energyid-fixedf} with $F \equiv 0$, we obtain that ${\mathcal P}_0(\overline U) \leq   {\mathcal E}(U;0)<0$. Hence, $\overline U$ is nontrivial. From \cite[Theorem 7]{2014al-gwaizNATMA} we know that the nontrivial solutions for \eqref{stationary} when the parameter $P$ is restricted to $(\Lambda_1, \Lambda_2]$ are $\pm\lambda_+ {}w_1{}$ and we infer that
\begin{empheq}{align}
{\mathcal P}_{0}(\overline U)=\inf_{w \in H^{2}_*(\Omega)} {\mathcal P}_{0}(w) \, ,
\end{empheq}
and hence that \eqref{condextra} is satisfied. Then, from \eqref{resultadoforte},
\begin{empheq}{align}
U(t) \to \overline U\text{ in }H^{2}_*(\Omega), \quad U'(t)\to 0\text{ in }L^2(\Omega), \quad \quad t\to\infty \, .
\end{empheq}
So, the solution $U$ induces an orbit $(U(t), U'(t))$ in $H^{2}_*(\Omega)\times L^2(\Omega)$ that starts at $(U_0,V_0)$, tends to the equilibrium point $(\overline{U},0)$ and the energy $\mathcal{E}$ is negative along this orbit. On the other hand, $\mathcal{E} \geq 0$ on $[{}w_1{}]^{\perp} \times L^2(\Omega)$, which implies that $\|{U(t)-\lambda_+ {}w_1{}}\|_2\neq\|{U(t)+\lambda_+ {}w_1{}}\|_2$ for all $t \geq 0$.  Therefore, if $\|{U_0-\lambda_+ {}w_1{}}\|_2^2<\|{U_0+\lambda_+ {}w_1{}}\|_2^2$ then this inequality remains true throughout the orbit and $\overline U = \lambda_+ {}w_1{}$, while if $\|{U_0-\lambda_+ {}w_1{}}\|_2^2> \|{U_0+\lambda_+ {}w_1{}}\|_2^2$ then we infer that $\overline U = -\lambda_+ {}w_1{}$.
\end{proof}

\begin{proof}[\textbf{Proof of Theorem \ref{thm-vkperp}}]
Let $U$ be a weak solution of \eqref{enlm}, $v\in H^{2}_*(\Omega)$ and set
\begin{empheq}{align}
\psi_v (t) = \langle{U(t),V}\rangle \, .
\end{empheq}
So, by the regularity of $U$, $\psi_v'$ and $\psi_v''$ are well defined,  and
\begin{empheq}{align}\psi_v'=\langle{U',V}\rangle \, , \quad \psi_v''= {\langle U'',V \rangle} \, ,\end{empheq}
the latter being perfect for plugging into \eqref{weakform}. We intend to derive a second order differential equation for $\psi_v$, so we must find meaning for the remaining terms. By Riesz representation Theorem, the action of $U(t)$ as an element of $\mathcal{H}$ is
\begin{empheq}{align}\langle{U,w}\rangle = ({U,w})_2, \quad \forall \, w\in H^{2}_*(\Omega)\, .\end{empheq}

If $v = V_k$ is an eigenfunction associated to $\Lambda_k$, and writing for short $\psi_k  = \psi_{V_k}$, then
\begin{empheq}{align}\label{contaspsik}
{\langle U'',V \rangle} + \delta ({U',V})_0 + (U,V)_2 +\phi(U) (U_x,V_x)_0&
=\psi_k''+ \delta({U',V_k})_0 +({U,V_k})_2\left( 1+\frac{\phi(U)}{\Lambda_k}\right)\\
& =\psi_k''+ \frac\delta {\lambda_k}\psi'_k +\left( 1+\frac{\phi(U)}{\Lambda_k}\right)\psi_k \, .
\end{empheq}
Indeed, we can write
\begin{empheq}{align}
\delta({U',V_k})_0 & = \delta \lim_{h \to 0}  \frac{({U(t +h),V_k})_0 - ({U(t),V_k})_0 }{h}
= \frac\delta {\lambda_k} \lim_{h \to 0} \frac{({U(t +h),V_k})_2 - ({U(t),V_k})_2 }{h} \\
& = \frac\delta {\lambda_k} \lim_{h \to 0} \frac{\langle{U(t +h),V_k}\rangle - \langle{U(t),V_k}\rangle }{h} = \frac\delta {\lambda_k}\psi'_k \, .
\end{empheq}
So, with $F \in [V_k]^{\perp}$, the ordinary differential equation for $\psi_{k}$ reads
\begin{empheq}{align}
\psi_k''&+  \frac\delta {\lambda_k}\psi'_k + \left( 1+\frac{\phi(U)}{\Lambda_k}\right)\psi_k= 0 \, .
\end{empheq}

Then observe that the weight $E_k(t) = 1+{\phi(U)} / {\Lambda_k}$ is continuous. Therefore, the initial value problem
\begin{empheq}{align}
\psi''&+  \frac\delta {\lambda_k}\psi' + \left( 1+\frac{\phi(U)}{\Lambda_k}\right)\psi= 0\,, \quad \quad \psi(0) = \psi'(0) =0 \, ,
\end{empheq}
has a unique solution, the zero solution. We conclude that whenever $F \in [V_k]^{\perp} \subset L^2(\Omega)$, $U_0 \in [V_k]^{\perp} \subset H^{2}_*(\Omega)$, $V_0 \in [V_k]^{\perp} \subset L^2(\Omega)$, then $U(t) \in [V_k]^{\perp} \subset H^{2}_*(\Omega)$ for all $t$.
\end{proof}

\medskip
\begin{proof}[\textbf{Proof of Corollary \ref{th:perpv1}}]

It is a straightforward consequence of Lemma \ref{convergencelemma} ii) and Theorem \ref{thm-vkperp}.
Indeed, by Theorem \ref{thm-vkperp}, $U(t) \in [{}w_1{}]^{\perp}$ for all $t \geq 0$ and so
\begin{empheq}{align}
{\mathcal P}_{F}(U;t) \geq 0  = {\mathcal P}_{F}(0) \ \ \forall \, t \geq 0.
\end{empheq}
Since $\Lambda_1 < P \le\Lambda_2$, $\overline U = 0$ is the unique solution of \eqref{stationary} in $[{}w_1{}]^{\perp}$.
\end{proof}

\section{Proof of stability/instability of modes}\label{section:stability}

Throughout this section we consider the  problem \eqref{enlm} under dynamical equilibrium, i.e. $F\equiv 0$, $\delta = 0$. First we will show that, fixing $m$ and $k$, the energy defines, up to translation in time, a unique simple mode.

\begin{proposition}\label{carac-simple-modes} Let $\delta =0$ and $P\leq{\Lambda_{m,i}}$. The family of $(m,i)$-simple  modes is parametrized, up to translation in time, by the energy
 and the simple modes are time-periodic.
\end{proposition}

Indeed, Proposition \ref{carac-simple-modes} holds in a more general setting.
Consider the ODE
\begin{empheq}[left=\empheqlbrace]{align}
w''+h(w) =0,\label{simpleode}\\
w(0)=a, \quad w'(0)=b,
\end{empheq}
where $h$ and its primitive $H(s) = \int_0^s h(\sigma)d\sigma$ satisfy
\begin{equation}\label{ff}
h(s)s>0, \quad s\neq 0,\qquad H(s) \to \infty, \quad |s|\to\infty.
\end{equation}

\begin{proposition}\label{lemma-general-periodic}
If $h$ is a locally Lipschitz function satisfying \eqref{ff}, then for all $(a,b)\in \mathbb{R}^2\backslash\{(0,0)\}$, the solution $w$ of \eqref{simpleode} is periodic and sign changing. Moreover, if two solutions have the same (conserved) energy
\begin{empheq}{align}
E(w,w')= \frac12w'^2 + H(w),
\end{empheq}
then, up to translation with respect to the variable $t$, they coincide. In addition, if $h$ is odd, then so are all of the solutions, with respect to its zeros.

\end{proposition}
\begin{proof}
Let $w$ be a solution of \eqref{simpleode} with initial data $(a,b)\neq(0,0)$. Since \eqref{simpleode} is a conservative equation, \eqref{ff} ensures that all solutions are bounded, and hence globally defined.

\vspace{10pt}

\noindent {\em Claim:}  $w$ is not eventually of one sign ($w\geq0$ or $w \leq0$ for all $t\geq t_0$).

\noindent Assume by contradiction that there exists $t_0>0$ such that $w\geq0$ on $[t_0,\infty)$. Then, since $w\not\equiv 0$, $w(t_0)>0$ or $w'(t_0)>0$. In any case, we infer that $w$ is positive in a neighborhood on the right of $t_0$. In addition, from the equation \eqref{simpleode} and \eqref{ff}, $w$ is concave and $w'$ is nonincreasing in this interval. If $w'(t_1)<0$ for some $t_1> t_0$, then we readily have
\begin{empheq}{align}
w(t) \leq w(t_1) +w'(t_1)(t-t_1), \quad t>t_1,
\end{empheq}
and $w$ changes sign in $[t_0,\infty)$, a contradiction. Assume then $w'\geq0$ in $(t_0,\infty)$. As $w$ remains bounded and concave, it follows that
\begin{empheq}{align}\label{Lcontradiciton}w(t) \to L>0, \quad w'(t)\to 0, \quad w''(t)\to -h(L) <0, \quad t\to\infty.\end{empheq}
Hence, from the mean value Theorem, as $w'\to0$,  we can build a sequence $(t_n)$ such that
\begin{empheq}{align}t_n\to\infty, \quad w''(t_n) \to 0,\end{empheq}
 which contradicts \eqref{Lcontradiciton}. Similarly, we can prove that $w$ is not eventually nonpositive.

\vspace{10pt}

\noindent {\em Claim:} $w$ is periodic.

\noindent Let $(\tau_n)$ be the sequence of zeros of $w$ in $[0,\infty)$. Then $w$ has opposite signs in $(\tau_1,\tau_2)$ and $(\tau_2,\tau_3)$ and $w(\tau_1)=w(\tau_3)=0$.  From the energy conservation
\begin{empheq}{align}(w'(\tau_1))^2=2 E(0,w'(\tau_1)) = 2 E(0,w'(\tau_3)) = (w'(\tau_3))^2.\end{empheq}
Therefore $w'(\tau_3) = w'(\tau_1)$ and hence $w$ is $T$-periodic with $T=\tau_3-\tau_1$.

\vspace{10pt}

\noindent {\em Claim:} The energy determines a unique solution, up to translation.

\noindent Let $(a,b),(c,d)\in\mathbb{R}^2$ be initial data such that
\begin{empheq}{align}\frac{b^2}{2}+ H(a)=E(a,b) = E(c,d) = \frac{d^2}{2}+ H(b).\end{empheq}
Denote by $w(a,b,\cdot)$ and $w(c,d, \cdot)$ the corresponding solutions. Let $\tau(a,b), \tau(c,d) \in {\mathbb{R}}$ such that
\begin{empheq}{align}
w(a,b,\tau(a,b))= 0, \,  w'(a,b,\tau(a,b))>0, \, w(c,d,\tau(c,d))= 0, \,  w'(c,d,\tau(c,d))>0.
\end{empheq}
Since
\begin{empheq}{align}
w'(a,b,\tau(a,b))^2 =& 2E(0,w'(a,b,\tau(a,b))) \\
=& 2E(a,b) = 2E(c,d)= 2E(0,w'(c,d, \tau(c,d)))=w'(c,d,\tau(c,d))^2,\end{empheq}
we infer that $w'(a,b,\tau(a,b)) = w'(c,d,\tau(c,d))$, and hence $w(a,b,t) = w(c,d,R+ t)$ for all $t$, where  $R=\tau(c,d)-\tau(a,b)$.

\vspace{10pt}

\noindent {\em Claim:} If $h$ is odd, then so are all of the solutions, with respect to its zeros.

\noindent Let $\tau$ be a zero of $w$ and set $w_1(t) = w(t + \tau)$, $w_{2}(t) = - w(\tau -t)$. Then observe that $w_1$ and $w_2$ solves
\begin{empheq}{align}
z'' + h(z) = 0, \quad z(0) =0, \quad z'(0) = w'(\tau).
\end{empheq}
By uniqueness, we infer that $w_1 =w_2$ and hence that $w$ is odd with respect to $\tau$.
\end{proof}

\begin{proof}[\textbf{Proof of Proposition {\rm{\ref{carac-simple-modes}}}}]
It is a consequence of Proposition \ref{lemma-general-periodic}
with
\begin{empheq}{align}h_{m,i,P,S}(s) =m^2[(\Lambda_{m,i}  -P) s+ Sm^2 s^3]. \end{empheq}
Observe that {$S>0$ is sufficient for \eqref{ff}, whereas }the condition $P\leq \Lambda_{m,i}$ guarantees \eqref{ff}.
\end{proof}

\begin{proof}[\textbf{Proof of Theorem {\rm{\ref{th:newstability}}}}]
Here we follow closely the ideas in \cite[Theorem 1.1]{1996cazenaveQAM}, with some slightly different arguments because system \eqref{systemwz} has three parameters $\mu, \gamma, \nu$ and the corresponding system in \cite{{1996cazenaveQAM}} has only two.

With the change of unknowns $\varphi\mapsto m\varphi$ and $\psi\mapsto n\psi$ the system \eqref{systemODE} reads
$$
\left\{\begin{array}{l}
\varphi''(t)+m^2(\Lambda_{m,i}-P)\varphi(t)+Sm^2\big(\varphi(t)^2+\psi(t)^2\big)\varphi(t)=0\, ,\vspace{5pt}\\
\psi''(t)+n^2(\Lambda_{n,k}-P)\psi(t)+Sn^2\big(\varphi(t)^2+\psi(t)^2\big)\psi(t)=0\ .
\end{array}\right.
$$
Then the substitutions $\varphi(t)\mapsto \varphi(\frac{t}{m\sqrt{S}})$ and $\psi(t)\mapsto\psi(\frac{t}{m\sqrt{S}})$ lead to
$$
\left\{ \begin{array}{l}
\varphi''(t)+\frac{\Lambda_{m,i}-P}{S}\varphi(t)+(\varphi(t)^2+\psi(t)^2)\varphi(t)=0,\vspace{5pt}\\
\psi''(t)+ \frac{n^2}{m^2}\frac{\Lambda_{n,k}-P}{S}\psi(t)+\frac{n^2}{m^2}(\varphi(t)^2+\psi(t)^2)\psi(t)=0.
\end{array}\right.
$$

Setting $\mu=\frac{\Lambda_{m,i}-P}{S}$, $\nu=\frac{\Lambda_{n,k}-P}{S}$, $\gamma=\frac{n^2}{m^2}$, we obtain
\begin{equation}\label{Ham}
\left\{\begin{array}{l}
\varphi''(t)+(\mu+\varphi(t)^2+\psi(t)^2)\varphi(t)=0, \vspace{5pt}\\
\psi''(t)+\gamma(\nu+\varphi(t)^2+\psi(t)^2)\psi(t)=0,
\end{array}\right.
\end{equation}
which resembles \cite[(1.4)]{1996cazenaveQAM}; observe that the condition $P< \min\{\Lambda_{m,i},  \Lambda_{n,k}\}$ guarantees $\mu> 0$ and $\nu>0$.

Fix an energy level $E_0>0$, $m,n,i,k\in{\mathbb{N}}$ and
set $b_* = \sqrt{2E_0}$.
Consider the linear operator $L: {\mathbb{R}}^2 \to {\mathbb{R}}^2$ defined by $L(c,d) = -(z(\theta), z'(\theta))$, where $\theta$ is the first zero of $w$ and
\begin{empheq}[left=\empheqlbrace]{align}\label{systemwz}
w''+  \mu w +w^3=0,\\
z''+ \gamma(\nu +w^2) z=0,\label{hill-nova},
\end{empheq}
with $w(0)=0,\ w'(0)=b_*, \ z(0)=c,\ z'(0)=d.$ Observe that systems \eqref{systemwz}-\eqref{hill-nova} and \eqref{Ham} are the same and that $w$ is the coordinate of an $(m,i)$-simple  mode with energy $E_0$. Therefore
\begin{equation}\label{estable-equiv}
\text{\emph{the proof is reduced to the study of stability of $z=0$ in equation \eqref{hill-nova}}}.
\end{equation}
By Floquet theory, this is equivalent to study the eigenvalues of operator $L$; see \cite[Theorem 2.89, p. 194]{chicone}.

Problem \eqref{systemwz} can be seen as a
regular perturbation of
the system
\begin{empheq}[left=\empheqlbrace]{align}\label{systemwzlimit}
\vartheta''+  \vartheta^3=0,\\
\zeta''+ \gamma\vartheta^2 \zeta=0,
\end{empheq}
with $\vartheta(0)=0,\ \vartheta'(0)=1,\ \zeta(0)=\bar c,\ \zeta'(0)=\bar d,$ for some $\bar c, \bar d \in {\mathbb{R}}$.
Indeed, consider the problem
\begin{empheq}[left=\empheqlbrace]{align}
w_\varepsilon''+  \varepsilon w_\varepsilon +w_\varepsilon^3=0, \label{systemwzepsilon}\\
z_\varepsilon''+ \gamma(\nu \varepsilon/\mu +w_\varepsilon^2) z_\varepsilon=0,
\end{empheq}
with $w_\varepsilon(0)=0,\ w_\varepsilon'(0)=1,\ z_\varepsilon(0)=\bar c,\ z_\varepsilon'(0)=\bar d$, which is clearly a regular perturbation of \eqref{systemwzlimit} as $\varepsilon\to 0$. On the other hand, with
\begin{equation}\label{smallepsilon}
\varepsilon = \frac{\mu}{b_*} = \frac{\mu}{\sqrt{2 E_0}},
\end{equation}
 the functions
\begin{empheq}{align}
w(t) = \sqrt{\mu/\varepsilon} \,  w_{\varepsilon}(t\sqrt{\mu/\varepsilon}), \quad z(t) = z_{\varepsilon}(t\sqrt{\mu/\varepsilon}),
\end{empheq}
solve system \eqref{systemwz} provided we choose $(\bar c, \bar d) = \left(c,  d \sqrt{\varepsilon/\mu} \right)$.

In analogy with the operator $L$, consider the linear operator $B_\gamma$ associated to \eqref{systemwzlimit} defined by $B_\gamma(\bar c,\bar d) = -(\zeta(\tau), \zeta'(\tau))$, where $\tau$ is the first positive zero of $\vartheta(c,d)$. Analogously, define the operator $B_{\gamma,\varepsilon}$ associated to \eqref{systemwzepsilon}.

Cazenave and Weissler \cite{1996cazenaveQAM} studied in details the properties of the operator $B_{\gamma}$ showing that the stability properties of the system \eqref{systemwzlimit} depend on the value of  $\gamma$. From \cite[Theorem 3.1 (iii)]{1996cazenaveQAM}, if $\gamma \in I_j$ for some $j\in{\mathbb{N}}$, then the eigenvalues  $\lambda_1, \lambda_2$ of the operator $B_\gamma$ are complex and $|\lambda_1|=|\lambda_2|=1$.
This property is shared by the eigenvalues of $B_{\gamma,\varepsilon}$ for small enough $\varepsilon$, that is,
\begin{empheq}{align}\label{complex-eigenvalues}
|\lambda_{1,\varepsilon}|=|\lambda_{2,\varepsilon}|=1;
\end{empheq}
see \cite[Lemma 3.3 (i)]{1996cazenaveQAM}.

On the other hand, according to \cite[Theorem 3.1 (i)-(ii)]{1996cazenaveQAM}, if $\gamma \in K_j$ for some $j\in{\mathbb{N}}$, then the eigenvalues  $\lambda_1, \lambda_2$ of the operator $B_\gamma$ have the form $\lambda_2 = \lambda_1^{-1}$, $0<|\lambda_1|<1<|\lambda_2|=|\lambda_1|^{-1}$.
By regular perturbation, for sufficiently small $\varepsilon_1>0$ and
for all $0<\varepsilon\leq \varepsilon_1$, the eigenvalues $\lambda_{1,\varepsilon}, \lambda_{2,\varepsilon}$ of $B_{\gamma,\varepsilon}$ satisfy
\begin{empheq}{align}\label{real-eigenvalues}
0<|\lambda_{1, \varepsilon}| < 1 < |\lambda_{2,\varepsilon}| \,.
\end{empheq}

Now recall that $B_{\gamma,\varepsilon} = L$ for $\varepsilon$ as in \eqref{smallepsilon}. The condition of small $\varepsilon$ is translated by \eqref{smallepsilon} as a requirement of large energy $E_0$ for system \eqref{Ham}.
Once the conserved energy of systems \eqref{systemODE} is the energy of \eqref{Ham} times the factor $S$, this is equivalent to ask large energy for the first.

If $\gamma \in I_j$ and the energy condition is fulfilled, then by \eqref{complex-eigenvalues} the $(m,i)$-simple  mode is $(n,k)$ linearly stable. On the other hand, if $\gamma \in K_j$ and the energy is large enough, \eqref{real-eigenvalues} implies that the $(m,i)$-simple  mode is $(n,k)$ linearly unstable.
\end{proof}

\begin{proof}[\textbf{Proof of Theorem {\rm{\ref{stable}}}}]
System \eqref{Ham} with initial data $\varphi(0)=\alpha>0$ and $\varphi'(0)=\psi(0)=\psi'(0)=0$ has the solution
$(\varphi,\psi)=(\varphi_\alpha,0)$ where $\varphi_\alpha$ is the solution of the Duffing equation
\begin{equation}\label{duffing}
\varphi_\alpha''(t)+\mu\varphi_\alpha(t)+\varphi_\alpha(t)^3=0,\quad\varphi_\alpha(0)=\alpha,\ \varphi_\alpha'(0)=0.
\end{equation}
Equation \eqref{duffing} has the conserved energy
$$E=\frac{(\varphi_\alpha')^2}{2}+\mu\frac{\varphi_\alpha^2}{2}+\frac{\varphi_\alpha^4}{4}\equiv\mu\frac{\alpha^2}{2}+\frac{\alpha^4}{4} \, .$$
From this, we infer that
\begin{equation}\label{ancoraenergia}
2(\varphi_\alpha')^2=(\varphi_\alpha^2+\Theta_+)(\Theta_--\varphi_\alpha^2)\qquad(\mbox{where }\Theta_\pm=\sqrt{\mu^2+4E}\pm\mu);
\end{equation}
we did not emphasize here the dependence of $\Theta_\pm$ on the energy $E$. From \eqref{ancoraenergia} we deduce that
\begin{equation}\label{alpha}
\alpha=\|\varphi_\alpha\|_\infty=\sqrt{\Theta_-}\, ,\qquad-\sqrt{\Theta_-}\le\varphi_\alpha(t)\le\sqrt{\Theta_-}\quad\forall t
\end{equation}
and $\varphi_\alpha$ oscillates in this range. If $\varphi_\alpha(t)$ solves \eqref{duffing}, then also $\varphi_\alpha(-t)$ solves the same problem: this shows that
the period $T=T(E)$ of the solution $\varphi_\alpha$ of \eqref{duffing} is the double of the length of an interval of monotonicity for $\varphi_\alpha$.
Then we have that $\varphi_\alpha(T/2)=-\alpha=-\sqrt{\Theta_-}$ and $\varphi_\alpha'(T/2)=0$, while $\varphi_\alpha'<0$ in $(0,T/2)$. By rewriting \eqref{ancoraenergia} as
$$
-\sqrt2 \varphi_\alpha'=\sqrt{(\varphi_\alpha^2+\Theta_+)(\Theta_--\varphi_\alpha^2)}\qquad\forall \,t\in\left(0,T/2\right)\, ,
$$
by separating variables, and upon integration over the time interval $(0,T/2)$ we obtain
$$\frac{T(E)}{2}=\sqrt2 \int_{-\sqrt{\Theta_-}}^{\sqrt{\Theta_-}}\frac{d\phi}{\sqrt{(\phi^2+\Theta_+)(\Theta_--\phi^2)}}\,.$$
Then, using the fact that the integrand is even with respect to $\phi$ and through a change of variables, we obtain
\begin{equation}\label{TE}
T(E)=4\sqrt2 \int_0^1\frac{d\phi}{\sqrt{(\Theta_-\phi^2+\Theta_+)(1-\phi^2)}}\, .
\end{equation}
Both the maps $E\mapsto\Theta_\pm(E)$ are continuous and increasing for $E\in[0,\infty)$: hence, $E\mapsto T(E)$ is strictly decreasing. Moreover,
an asymptotic expansion shows that
\begin{equation}\label{asymptotic}
\Theta_+(E)=2\mu+\tfrac{2}{\mu}\, E+o(E)\ ,\quad \Theta_-(E)=\tfrac{2}{\mu}\, E+o(E)\ ,\qquad\mbox{as }E\to0\, .
\end{equation}
In turn, from \eqref{alpha} and \eqref{asymptotic} we infer that
\begin{equation}\label{asymptotic2}
\|\varphi_\alpha\|_\infty^2=\tfrac{2}{\mu}\, E+o(E)\ ,\qquad\mbox{as }E\to0\, ,
\end{equation}
while from \eqref{TE} and \eqref{asymptotic} we infer that
$$
T(E)=\frac{2\pi}{\sqrt{\mu}}\left(1-\frac{3E}{4\mu^2}\right)+o(E)\ ,\qquad\mbox{as }E\to0
$$
and, finally,
\begin{equation}\label{asymptotic3}
\left(\frac{2\pi}{T(E)}\right)^2=\mu\left(1+\frac{3E}{2\mu^2}\right)+o(E)\ ,\qquad\mbox{as }E\to0\, .
\end{equation}

According to Definition \ref{defstab*} and as we have pointed out in \eqref{estable-equiv}, the simple mode $\varphi(t)w_{mi}(x,y)$ is linearly stable with respect to
$\psi(t)w_{nk}(x,y)$ (see \eqref{systemODE}) if $\xi\equiv0$ is a stable solution of the linear Hill equation
\begin{equation}\label{hill2}
\xi''+a(t)\xi=0\, ,\qquad a(t)=\gamma(\nu+\varphi_\alpha(t)^2)\quad \forall \, t\, .
\end{equation}
Since $T=T(E)$ is the period of $\varphi_\alpha$, the function $a(t)$ in \eqref{hill2} has period $T(E)/2$.
Moreover, by \eqref{asymptotic2}, we know that
\begin{equation}\label{doppiastima}
\gamma\nu\le a(t)\le\gamma(\nu+\tfrac{2}{\mu}\, E)+o(E)\ ,\qquad\forall \, t\ge0\ ,\quad\mbox{as }E\to0\, .
\end{equation}

At this point, we distinguish two cases.

\vspace{10pt}
\noindent {\em Case 1.} $\sqrt{\gamma\nu/\mu}\not\in\mathbb{N}$.\par\noindent
Then we denote by $h$ the integer part of $\sqrt{\gamma\nu/\mu}$ (possibly zero) so that
$$
h^2\mu<\gamma\nu<(h+1)^2\mu\, .
$$
Moreover, by \eqref{doppiastima} we know that $a(t)\to\gamma\nu$ uniformly as $E\to0$, whereas by \eqref{asymptotic3} we know that
$$\left(\frac{2\pi}{T(E)}\right)^2\to\mu\qquad\mbox{as }E\to0\, .$$
By combining these three facts and by continuity, we infer that there exists $\overline{E}>0$ such that, if $E\le\overline{E}$ then
\begin{equation}\label{doppiaa}
h^2\left(\frac{2\pi}{T(E)}\right)^2\le a(t)\le (h+1)^2\left(\frac{2\pi}{T(E)}\right)^2\qquad\forall \, t\, .
\end{equation}
Then a criterion of Zhukovskii \cite{zhk} (see also \cite[Chapter VIII]{yakubovich}) states that $\xi\equiv0$ is stable for \eqref{hill2}.
Hence, $\varphi(t)w_{mi}(x,y)$ is linearly stable with respect to $\psi(t)w_{nk}(x,y)$.

\vspace{10pt}
\noindent
{\em Case 2.} $\sqrt{\gamma\nu/\mu}\in\mathbb{N}$.\par\noindent
Then we denote by $h+1=\sqrt{\gamma\nu/\mu}\in\mathbb{N}$ (zero excluded) so that $\gamma\nu=(h+1)^2\mu$. By uniform convergence and by continuity
the left inequality in \eqref{doppiaa} still holds provided that $E$ is sufficiently small. On the other hand, by \eqref{asymptotic3} and \eqref{doppiastima},
the right inequality in \eqref{doppiaa} holds provided that
$$\gamma(\nu+\tfrac{2}{\mu}\, E)+o(E)\le(h+1)^2\mu\left(1+\frac{3E}{2\mu^2}\right)+o(E)$$
for sufficiently small $E$. Since $\gamma\nu=(h+1)^2\mu$, the latter is certainly satisfied if $4\gamma<3(h+1)^2$, that is, if $\gamma<\frac34$.
Recalling the definition of $\gamma=n^2/m^2$ and assumption \eqref{nm}, we see that this holds since $(h+1)^2\ge1$.
Therefore, also the left inequality in \eqref{doppiaa} holds if $E$ is sufficiently small and we conclude as in Case 1.
\end{proof}

\section*{Acknowledgements}
The Authors are grateful to an anonymous Referee for several remarks which lead to an improvement of the present paper. Moreover, the Referee suggested
the following interesting variant of the equation in \eqref{enlm}:
$$U_{tt}+\delta U_{t}+\Delta ^{2}U-\phi(U) U_{xx}=F+GU_{y}\, .$$
The additional term $GU_y$ well describes the strong wind attacking the free edges of the plate and its effect would be a shift in the torsional
eigenvalues depending on the force $G$ of the wind.\par
V. Ferreira Jr is supported by FAPESP \#2012/23741-3 grant. F. Gazzola is partially supported by the PRIN project Equazioni alle derivate parziali di tipo ellittico e parabolico: aspetti geometrici, disuguaglianze collegate, e applicazioni, and he is member of the Gruppo Nazionale per l' Analisi Matematica, la Probabilit\`a e le loro Applicazioni (GNAMPA) of the Istituto Nazionale di Alta Matematica (INdAM). E. Moreira dos Santos was partially supported by CNPq \#309291/2012-7 and \#307358/2015-1 grants and FAPESP \#2014/03805-2 and \#2015/17096-6 grants.

\bibliographystyle{abbrv}

\end{document}